\documentclass[12pt, reqno]{amsart}

\usepackage{textcomp} 
\usepackage{amsmath, amsfonts,amsthm,amssymb,amsxtra}
\usepackage{amscd}
\usepackage{enumitem}
\usepackage{subcaption}

\usepackage{tikz-cd} 
\usepackage{todonotes}

\usepackage{graphicx}
\usepackage{verbatim}
\usepackage{hyperref}
\hypersetup{
    unicode=false,          
    pdftoolbar=true,        
    pdfmenubar=true,        
    pdffitwindow=false,     
    pdfstartview={FitH},    
    pdftitle={},    
    pdfauthor={Author},     
    pdfnewwindow=true,      
    colorlinks=true,       
    linkcolor=red,          
    citecolor=blue,        
    filecolor=magenta,      
    urlcolor=cyan,           
}

\usepackage{xcolor}
\definecolor{mygreen}{RGB}{0,220,0}
\definecolor{myorange}{RGB}{255,160,0}

\newtheorem{theorem}{Theorem}[section]
\newtheorem{proposition}[theorem]{Proposition}
\newtheorem{lemma}[theorem]{Lemma}
\newtheorem{corollary}[theorem]{Corollary}
\newtheorem{question}[theorem]{Question}
\newtheorem{conjecture}[theorem]{Conjecture}

\newtheorem*{cor:pi1}{Corollary~\ref{cor:pi1}}
\newtheorem*{cor:L0c1}{Corollary~\ref{cor:L0c1}}
\newtheorem*{cor:even}{Corollary~\ref{cor:even}}
\newtheorem*{thm:connected}{Theorem~\ref{thm:connected}}
\newtheorem*{theorem:PtoLF}{Theorem~\ref{thm:PtoLF}}
\newtheorem*{theorem:WeinsteinHasPath}{Theorem~\ref{thm:WeinsteinHasPath}}
\newtheorem*{theorem:loopssmooth4viaALF}{Theorem~\ref{loopssmooth4viaALF}}

\theoremstyle{definition}

\newtheorem{definition}[theorem]{Definition}
\newtheorem{examples}[theorem]{Example}

\newtheorem{procedure}[theorem]{Procedure}

\theoremstyle{remark}

\newtheorem{remark}[theorem]{Remark}

\numberwithin{equation}{section}

\newcommand{\C}{\mathbb{C}}

\renewcommand{\phi}{\varphi}
\renewcommand{\epsilon}{\varepsilon}
\newcommand{\ob}{\mathcal{OB}}

\newcommand{\cL}{\mathcal{L}}
\renewcommand{\P}{\mathbb{P}}
\newcommand{\cptwo}{\C\P^2}

\newcommand{\cptwobar}{\overline{\C\P}\,\!^2}

\begin{document}
	\title{The contact cut graph and a Weinstein $\mathcal{L}$-invariant}
	
	\begin{abstract}
 We define and study the {\em contact cut graph} which is an analogue of Hatcher and Thurston's cut graph for contact geometry, inspired by contact Heegaard splittings~\cite{Tor00,Gir02}. We show how oriented paths in the contact cut graph correspond to Lefschetz fibrations and multisection with divides diagrams. We also give a correspondence for achiral Lefschetz fibrations. We use these correspondences to define a new invariant of Weinstein domains, the {\em Weinstein $\cL$-invariant}, that is a symplectic analogue of the Kirby-Thompson's $\cL$-invariant of smooth $4$-manifolds. We discuss the relation of Lefschetz stabilization with the Weinstein $\cL$-invariant. We present topological and geometric constraints of Weinstein domains with $\cL=0$. 
We also give two families of examples of multisections with divides that have arbitrarily large $\cL$-invariant.  
	\end{abstract}
	
	

\author[N.\ Castro]{Nickolas A. Castro}
\address{Rice University
Mathematics Department,
Houston, TX 77005-1892, USA}
\email{ncastro.math@rice.edu}\urladdr{https://nickcastromath.com}

\author[G.\ Islambouli]{Gabriel Islambouli}
\address{University of California Davis, Mathematics Department,
Davis, CA 95616, USA}
\email{islambouli@ucdavis.edu}\urladdr{https://sites.google.com/view/gabrielislambouli/}

\author[J.\ Min]{Jie Min}
\address{Department of Mathematics and Statistics,
University of Massachusetts Amherst, MA 01003-9305, USA}
\email{jmin@umass.edu}\urladdr{https://sites.google.com/site/jieminmath/home}

\author[S.\ Sakall{\i}]{S\"umeyra Sakall{\i}}
\address{Department of Mathematics and Statistics,
University of South Florida,
Tampa, FL, 33620, USA}
\email{smyrasa@gmail.com}\urladdr{https://sites.google.com/view/smyrsakalli/home}

\author[L.\ Starkston]{Laura Starkston}
\address{University of California Davis, Mathematics Department,
Davis, CA 95616, USA}
\email{lstarkston@math.ucdavis.edu}\urladdr{https://www.math.ucdavis.edu/~lstarkston/}

\author[A.\ Wu]{Angela Wu}
\address{Department of Mathematics,
Bucknell University,
Lewisburg, PA 17837,
USA}
\email{a.wu@bucknell.edu}\urladdr{https://angelamath.com}
	
	\maketitle
	

	\section{Introduction}
	
	There is a long history of using graphs associated to a surface in order to understand manifolds. Hempel utilized the curve complex to give criteria for the reducibility and weak reducibility of Heegaard splittings of 3-manifolds~\cite{Hem01}. Prior to this, Hatcher and Thurston \cite{HT20} used a related complex, which they called the cut complex, in order to show that the mapping class group is finitely presented. The key result they proved in order to produce a presentation for the mapping class group was that the cut graph is connected (and the cut complex is simply connected). Drawing on these prior works, Kirby and Thompson \cite{KT18} showed that every trisection of a smooth 4-manifold, and hence every smooth 4-manifold, could be described as a loop in the cut complex of a surface. Using this description, they defined an invariant of smooth orientable closed 4-manifolds, which roughly measured the shortest such loop corresponding to the given 4-manifold. Since their work, these ideas have been expanded to provide invariants of 4-manifolds with boundary \cite{CIMT21}, knotted surfaces in the 4-sphere and 4-ball~\cite{APTZ21,BCTT22}, and knotted surfaces in arbitrary 4-manifolds~\cite{ABGKKMP23}. 
	
	Our goal in this article is to define and study an analogue of this in the contact and symplectic setting. A cut system can encode a standard contact structure on the handlebody, provided the surface $\Sigma$ is endowed with the additional information of a dividing set $d$ which cuts $\Sigma$ into two homeomorphic surfaces with boundary. A cut system must intersect the dividing set in a particular way in order to realize a standard contact structure on the handlebody. Fixing $(\Sigma,d)$, we identify a certain subset of vertices of the cut graph which are ``contact'' as well as a subset of geometrically meaningful edges. This leads us to a contact analogue of the cut graph $CC(\Sigma,d)$ (see Section~\ref{s:ccdef} for the precise definition). Our first result is the analogue of the Hatcher-Thurston result in this contact setting.
	
	\begin{thm:connected}
		The contact cut graph $CC(\Sigma,d)$ is connected.
	\end{thm:connected}

	Our motivation for studying $CC(\Sigma,d)$ comes from an interest in defining and studying new invariants of Weinstein $4$-manifolds.
	Recently in \cite{IslambouliStarkston} the authors showed that every Weinstein 4-manifold admits a trisection-like structure called a \emph{multisection with divides}. The diagram for a multisection with divides lies on a surface $\Sigma$ with dividing set $d$ and consists of a sequence of vertices in $CC(\Sigma,d)$. 
	
	In Section~\ref{s:LF}, we develop a procedure to go from an oriented path $P$ in our contact cut graph $CC(\Sigma,d)$ to a multisection with divides diagram $\mathfrak{MD}(P)$ compatible with a Lefschetz fibration $LF(P)$ for a Weinstein domain $W(P)$.
	
	
	\begin{theorem:PtoLF}
		Let $P$ be an oriented (resp. unoriented) path in $CC(\Sigma, d)$. There exists an (achiral) allowable Lefschetz fibration $LF(P)$ with fiber $\Sigma^+$ such that $\mathfrak{MD}(P)$ is a diagram for the (achiral) multisection with divides compatible with $LF(P)$.
	\end{theorem:PtoLF}

	Next, we show that every Weinstein domain is obtained in this way.

	\begin{theorem:WeinsteinHasPath} 
		For any Weinstein domain $(W,\omega, V)$ (up to Weinstein homotopy), there exists a pair $(\Sigma,d)$ and path $P\in CC(\Sigma,d)$ such that $(W,\omega, V)=W(P)$.
	\end{theorem:WeinsteinHasPath}

	By considering unoriented paths, we get similar \emph{achiral} multisections with divides, which are no longer compatible with a global symplectic structure. Losing the symplectic structure allows us to understand all smooth 4-manifolds (in the complement of some embedded circles) through loops in a contact cut graph.
	
	\begin{theorem:loopssmooth4viaALF}
		For every closed, oriented smooth 4-manifold $X$, there exists a regular neighborhood $N$ of disjointly embedded circles such that $X\setminus N$ can be expressed as $X_C(L)$ for some unoriented loop $L$ in $CC(\Sigma, d)$.
	\end{theorem:loopssmooth4viaALF}
Here $X_C(L)$ denotes the closed 4-manifold obtained by capping off the boundary components of the $n$-section $X(L)$ associated to the loop $L$ with $\sqcup_g S^1 \times B^3$ (see Definition \ref{def:XCP}).	

	Using these correspondences, we define a new invariant of $4$-dimensional Weinstein domains, the \emph{Weinstein $\cL$-invariant}. This is a symplectic analogue of the Kirby-Thompson $\cL$-invariant of smooth $4$-manifolds. In the smooth setting, Kirby and Thompson proved that the only closed manifolds $X$ with $\cL(X)=0$ are connect sums of $\cptwo$, $\cptwobar$, $S^2\times S^2$, $S^1\times S^3$, and $S^4$~\cite[Theorem 12]{KT18}.  For Weinstein domains, there is no such decomposition of $\cL(W,\omega)=0$ domains into a finite list of basic pieces. Nevertheless, we can still prove significant constraints on their topology and geometry.
	\begin{cor:pi1}
		If $(W,\omega,V)$ is a Weinstein domain with $\cL = 0$, then $\pi_1(W)$ is a free group.
	\end{cor:pi1}
	\begin{cor:L0c1}
		If $(W,\omega,V)$ is a Weinstein domain with $\cL = 0$, then $c_1(\omega)=0$.
	\end{cor:L0c1}
	\begin{cor:even}
		If $(W,\omega,V)$ is a Weinstein domain with $\cL = 0$, then the intersection form on $H_2(W)$ is even.
	\end{cor:even}
	See Theorem~\ref{thm:L0Weinsteinhandles} for more precise constraints which imply these corollaries.
	
	We also look for Weinstein domains with arbitrarily large $\cL$-invariant. While it is difficult to consider all possible multisections with divides for a given Weinstein domain, we can get lower bounds on the length of a path compatible with a specific multisection with divides. In Section~\ref{ss:largeL}, we give two families of Weinstein domains $W_n$ and $W_n'$, with multisections with divides $\mathfrak{M}_n$ and $\mathfrak{M}_n'$ such that $\cL(\mathfrak{M}_n) = \cL(\mathfrak{M}_{n+1}')=n$. We suspect that these paths are actually minimal for the underlying Weinstein domains, not just the specific multisections with divides. See Section~\ref{ss:largeL}.

\subsection{Outline}
In Section \ref{s:Background} we review background on the classical cut graph, multisections, contact Heegaard splittings, and multisections with divides. In Section \ref{s:ccdef}, we give our definition of the contact cut graph and establish some important properties. In Section \ref{s:Conn}, we prove that our contact cut graph is connected. In Section~\ref{s:LF}, we explain how oriented paths in the contact cut graph correspond to Lefschetz fibrations and multisection with divides diagrams. We also give a correspondence for achiral Lefschetz fibrations. We use these correspondences to define a new invariant of Weinstein domains, the Weinstein $\cL$-invariant. In Section \ref{s:Stab}, we discuss how stabilization of a multisection with divides fits in with the contact cut graph and Weinstein $\cL$-invariant. In Section \ref{s:L0} we prove results about Weinstein domains with $\cL=0$, and finally in Section \ref{ss:largeL} we provide two families of examples that appear to have arbitrarily large $\cL$-invariant.

\subsection{Acknowledgements}
This project started at the Trisectors Workshop 2023: Connections with Symplectic Topology at the University of California, Davis in June, 2023, supported by NSF grant DMS 2308782. The authors would like to thank Jeff Meier, Maggie Miller, and Alex Zupan for organizing this workshop with LS and GI. The authors would also like to thank Austin Christian for helpful discussions. GI was supported by NSF DMS 1904074. LS was supported by NSF CAREER grant DMS 2042345 and a Sloan Fellowship. AW was supported by NSF grant DMS-2238131. 

\section{Background} \label{s:Background} 

The cut graph and cut complex $C(\Sigma)$ were defined by Hatcher and Thurston~\cite{HT80} for the purpose of studying the mapping class group of surfaces. The vertices correspond to cut systems: collections of curves in the surface which define a handlebody filling the surface. The edges correspond to intersection relations between two cut systems. Later Kirby and Thompson related paths in the cut complex to smooth 4-manifolds via trisections and used the length of a path in the cut complex to define the $\cL$-invariant of smooth 4-manifolds which can detect the 4-sphere among homology 4-spheres~\cite{KirTho18}. In this section, we review the relevant precise definitions which motivate our contact cut graph.

	\subsection{Classical cut graph and multisections}

	\begin{definition}
		A \emph{cut system} for a closed oriented genus $ g $ surface $ \Sigma_g $ is an unordered collection $ \{ \alpha_1,\dots,\alpha_g \} $ of $ g $ disjoint simple closed curves on $ \Sigma_g $ that cut $ \Sigma $ open into a $ 2g $-punctured sphere.
	\end{definition}

	A cut system defines a $3$-dimensional handlebody $H$ with boundary $\Sigma_g$, by specifying curves in $\Sigma_g$ which bound disks in $H$. If we take $\Sigma_g\times [0,\epsilon]$ and glue $D^2\times [-\delta,\delta]$ along a neighborhood of each curve in the cut system in $\Sigma_g\times\{\varepsilon\}$, we obtain a 3-dimensional compression body with two boundary components: $\Sigma_g\times\{0\}$ and a component homeomorphic to $S^2$ (using the condition on the cut system). Filling in $S^2$ by $B^3$ (which can be done in a unique way) we obtain the $3$-dimensional handlebody.

	\begin{figure}
	\centering
	\includegraphics[scale=0.2]{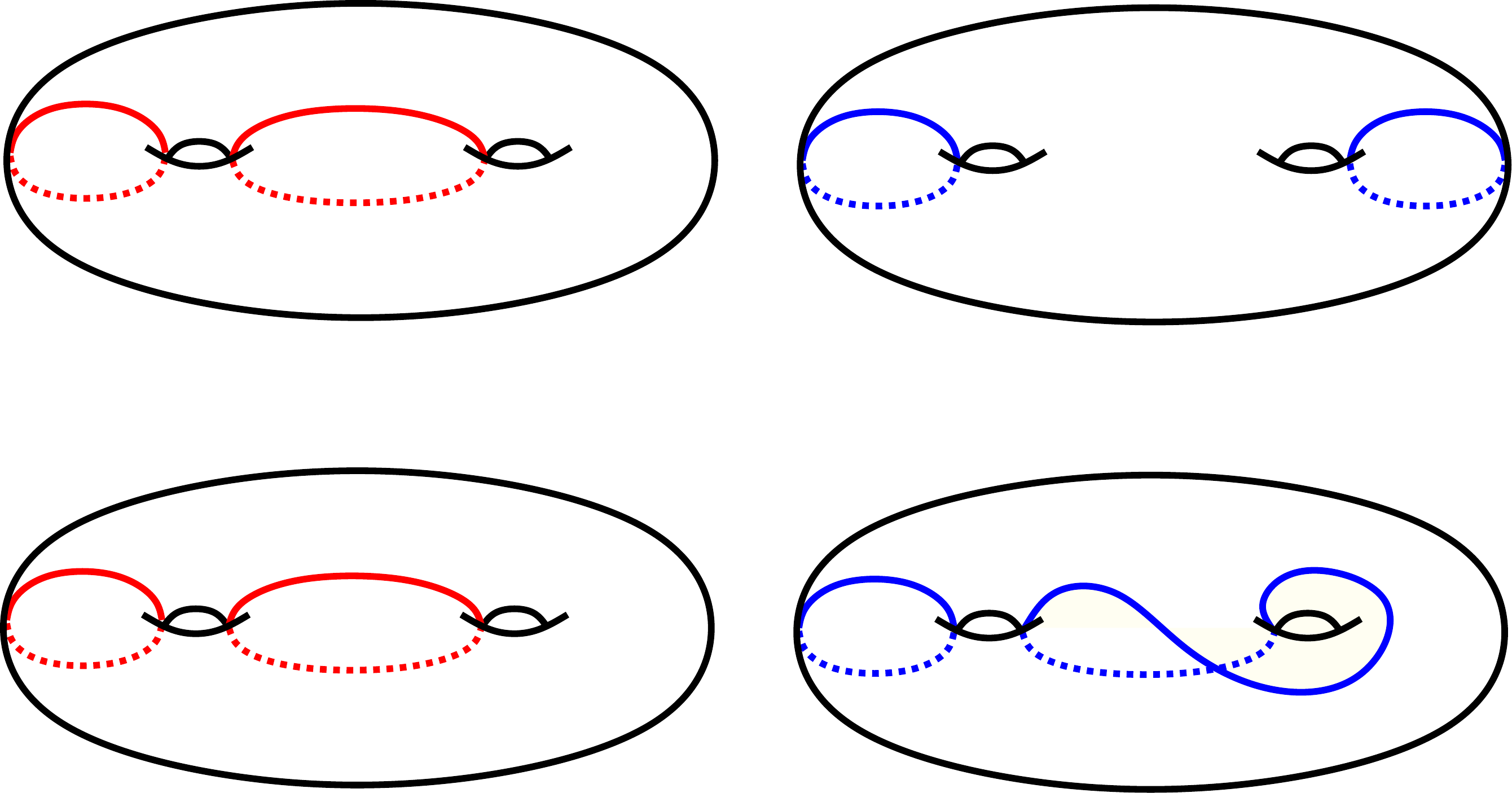}
	\caption{Top: the red and blue cut systems differ by a type 0 move. \\
	Bottom: the red and blue cut systems differ by a type 1 move.}
	\label{fig:cut_system_example}
	\end{figure}

	\begin{definition}
		Given a closed oriented genus $ g $ surface $ \Sigma_g $, the \emph{cut graph} $ C(\Sigma_g) $ is a $1$-complex, where the vertices correspond to isotopy classes of cut systems on $ \Sigma_g $. For two vertices $ v $ and $ v' $ with corresponding cut systems $ v=\{\alpha_1,\dots,\alpha_{g-1},\alpha_g \} $ and $ v'=\{ \alpha_1',\dots,\alpha_{g-1}',\alpha'_g \} $, they are connected by an \emph{edge of type $0$} if, up to reindexing, $\alpha_i=\alpha_i'$ for $i\le g-1$ and $ \alpha_g \cap \alpha_g'=\emptyset $. Similarly, they are connected by an \emph{edge of type $ 1 $} if, up to reindexing, $\alpha_i=\alpha_i'$ for $i\le g-1$ and $ \alpha_g $ intersects $ \alpha_g' $ transversely at a single point. See Figure~\ref{fig:cut_system_example} for an example.
	\end{definition}
	If two cut systems are related by a type $0$ move, their corresponding 3-dimensional handlebodies coincide, because $\alpha_g$ and $\alpha_g'$ are related by handleslides. Conversely, if two cut systems define the same handlebody, there is a sequence of type $0$ moves relating them. Type $1$ moves change the underlying handlebody.
	
	A multisection is a decomposition of a $4$-manifold into a union of sectors $X_1,\dots, X_n$ such that each $X_i \cong \natural_{k_i} S^1\times D^3$, consecutive sectors meet along $3$-dimensional handlebodies $X_i\cap X_{i+1} = H_i\cong \natural_g S^1\times D^2$, and all the sectors intersect along a surface which is the boundary of these handlebodies $X_1\cap \dots \cap X_n   = \Sigma_g = \partial H_i$ for all $i$. When the $4$-manifold has no boundary, the first and last sectors intersect along a 3-dimensional handlebody as well $X_n\cap X_1 = H_n$, whereas when the $4$-manifold has boundary, the first and last sectors only intersect along $\Sigma$, and we require $\partial X_1\setminus H_1 =: H_0$ together with $\partial X_n\setminus H_{n-1}=:H_n$ to be a Heegaard splitting of the boundary of the 4-manifold. Multisections without boundary generalize the trisections of Gay and Kirby~\cite{GayKir16}. Multisections were defined and interpreted as paths and loops in the cut graph in~\cite{IslNay20}.

	A vertex in the cut graph, $v$, represents a handlebody $H_v$. 
	If $v$ and $w$ are related by a type 0 move in $C(\Sigma_g)$, then $H_v \cup_{\Sigma_g} H_w$ is Heegaard splitting of $\#^{g} S^1 \times S^2$. We denote $X_{v,w}$ to be $\natural^{g} S^1 \times B^3$, whose boundary is identified with $H_v \cup H_w$. Similarly, if they are related by a type 1 move, then $H_v \cup_{\Sigma_g} H_w$ is Heegaard splitting of $\#^{g-1} S^1 \times S^2$, so we glue in $X_{v,w}\cong \natural^{g-1} S^1 \times B^3$.
	
	\begin{definition}\label{def:XCP}
		Let $P = (v_1, v_2... v_n)$ be a path in $C(\Sigma_g)$. The \emph{$n$-section associated to $P$} is the $4$-manifold with the decomposition $X(P) = \cup_{i=1}^{n-1} X_{v_i, v_{i+1}}$ where $X_{v_{i-1}, v_{i}}$ is glued to $X_{v_{i}, v_{i+1}}$ along $H_{v_i}$. If $v_1 = v_n$ so that $P$ is a loop, then we define $X_C(P)$ to be the closed 4-manifold obtained by capping off the boundary components of $X(P)$ with $\sqcup_g S^1 \times B^3$.
	\end{definition}
	
	\subsection{Multisections with divides and contact handlebodies}
	
	In this paper, we will look at an analogue of the cut graph which takes into account contact geometry. To set this up, we will need to define the analog of a cut system for a surface with boundary.
	\begin{definition}
	An \emph{arc system} for a genus $p$ surface $F$ with $b>0$ boundary components is a collection of $g=2p+b-1$ disjoint, properly embedded arcs $\mathcal{A} = \{a_1, \ldots, a_g\}$ such that $\overline{F\setminus \mathcal{A}}$ is a disc.
	\end{definition}
	One way to get from one arc system to another for a fixed surface with boundary is via arc slides.

	\begin{figure}
		\centering
		\includegraphics[scale=0.2]{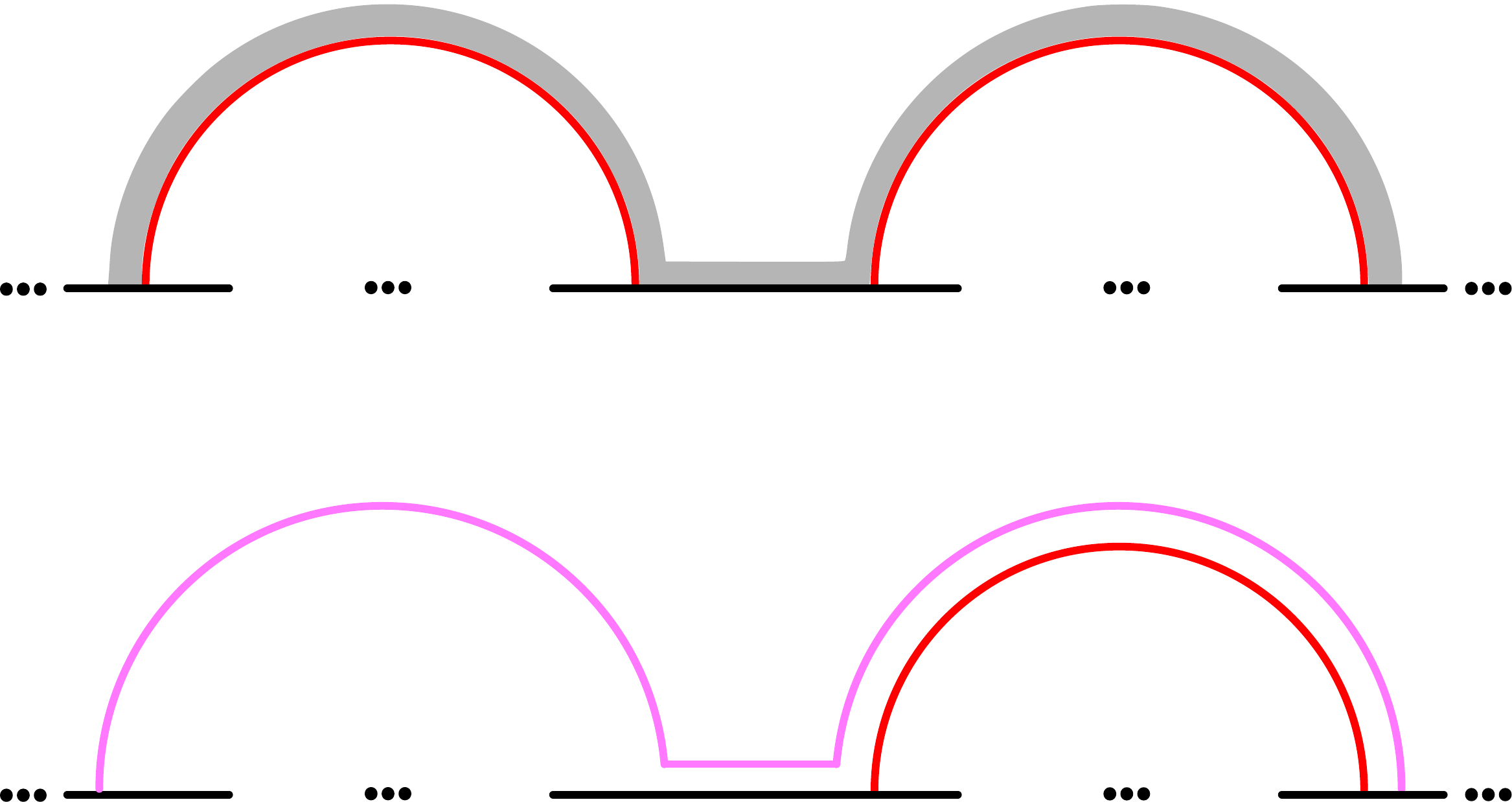}
		\caption{An arcslide.}
		\label{fig:arcslide1}
	\end{figure}
	
	\begin{definition}
		Let $\mathcal{A} = \{a_1, \ldots, a_g\}$ be a complete cut sytem of arcs for $S$. Let $\gamma$ be an embedded path contained in the boundary of $S$ starting at an end point, $a_i^+$ of $a_i$ and ending at an end point $a_j^-$ of $a_j$ and such that the interior of $\gamma$ is disjoint from $\mathcal{A}$. Let $N$ be a small regular neighborhood of $a_i\cup\gamma\cup a_j$ which is disjoint from the rest of $\mathcal{A}$. Let $a_i^-$ and $a_j^+$ be the other endpoints of $a_i$ and $a_j$ which are disjoint from $\gamma$. The \emph{arc slide} of $a_i$ over $a_j$ along $\gamma$ is a new complete cut system of arcs $\mathcal{A}'=\{a_1,\ldots, a_g\}\setminus \{a_i\} \cup \{a_i'\}$ where $a_i'$ is the component of $\partial N$ which is isotopic to the concatenation of $a_i$, $\gamma$ and $a_j$. See Figure~\ref{fig:arcslide1}.
	\end{definition}

	If $\mathcal{A}$ is an arc system, we can check that $\mathcal{A}'$ is also an arc system as follows. $\overline{F\setminus \mathcal{A}'}$ is obtained by cutting $\overline{F\setminus \mathcal{A}}$ along $a_i'$ and gluing along $a_i$. By assumption $\overline{F\setminus \mathcal{A}}$ is a disk. The subset $N$ is a subdisk cobounded by $a_i\cup \gamma\cup a_j$ and $a_i'$. Since $a_i\cup \gamma\cup a_j$ is in the boundary of $\overline{F\setminus\mathcal{A}}$, cutting along $a_i'$ creates two disks. Since each of these two disks has a copy of $a_i$ in its boundary, gluing along $a_i$ yields a disk again. See Figure~\ref{fig:arcslide2}.
	
	\begin{figure}
		\centering
		\includegraphics[scale=0.2]{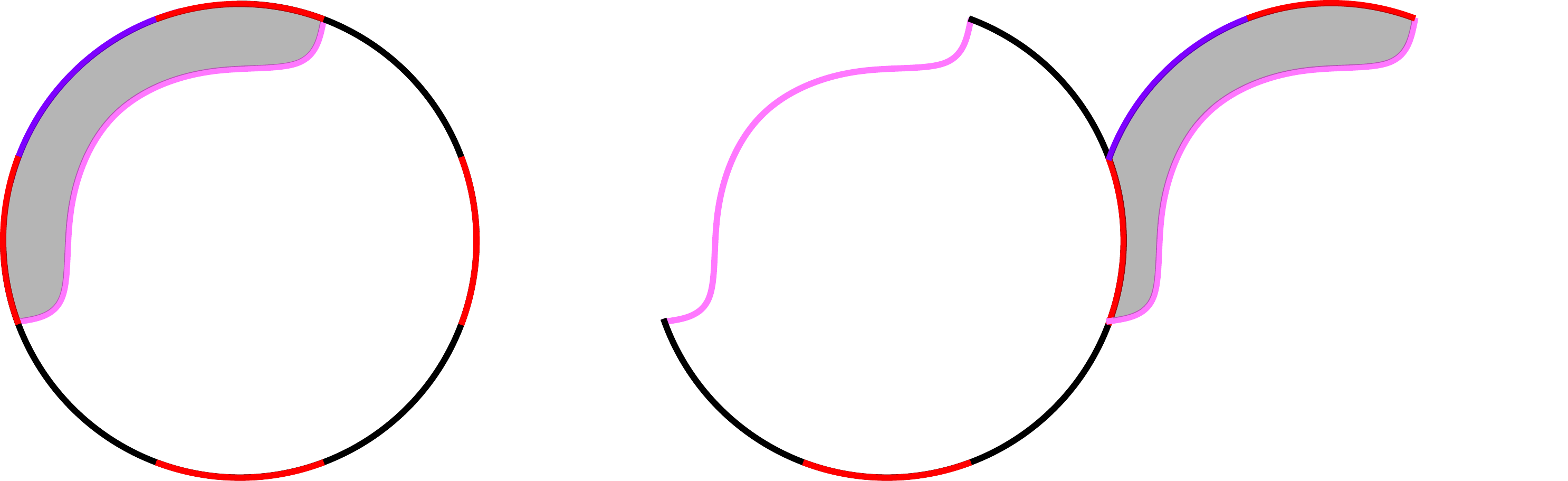}
		\caption{The result of performing an arcslide on an arc system is another arc system.}
		\label{fig:arcslide2}
	\end{figure}

	A Heegaard splitting $H_1\cup H_2$ of a contact 3-manifold $(Y,\xi)$ is a \emph{contact Heegaard splitting}~\cite{Tor00,Gir02} if $(H_i,\xi)$ is a standard contact handlebody for $i=1,2$. There are a few equivalent criteria to define/recognize when $(H_i,\xi_i)$ is a standard contact handlebody. The most useful characterization for our purposes is, $\Sigma :=\partial H_i$ is a convex surface with dividing set $d$ cutting $\Sigma$ into two pieces $\Sigma^+$ and $\Sigma^-$ which are diffeomorphic surfaces with boundary, and there exists a collection of compressing disks for $H_i$ such that their boundaries give a cut system $C^i$ for $\Sigma$ such that $C^i\cap \Sigma^\pm$ is an arc systems for $\Sigma^\pm$. In particular, each $C^i$ must intersect the dividing set in exactly two points (the end points of the arc).

	\emph{Multisections with divides} were introduced in~\cite{IslambouliStarkston} to capture the symplectic structure of a Weinstein domain diagrammatically via multisections. They are multisections of Weinstein domains (which necessarily have boundary), where each of the pieces ($X_i$ and $H_i$) of the multisection has a standard symplectic or contact structure. A standard symplectic structure on $\natural_k S^1\times D^3$ comes from Weinstein's handle construction~\cite{Wei91} attaching $4$-dimensional $1$-handles to a single $0$-handle. This comes with a Liouville structure which yields a contact structure on the boundary $\#_k S^1\times S^2$. The symplectic $\natural_k S^1\times D^3$ is a symplectic filling of this contact boundary, so in particular the contact structure is tight. Note there is a unique tight contact structure on $\#_k S^1\times S^2$. To make each of the 3-dimensional pieces standard, we also want this contact structure to be compatible with the Heegaard splitting. There are three equivalent characterizations of contact Heegaard splittings:
	\begin{enumerate}
		\item Each contact handlebody in the Heegaard splitting is a standard neighborhood of a Legendrian graph. 
		\item The Heegaard splitting comes from an open book decomposition. 
		\item The Heegaard surface is a convex surface with dividing set $d$ which cuts it into two homeomorphic pieces $\Sigma^+$ and $\Sigma^-$, and there exists a cut system for each handlebody which intersects $\Sigma^\pm$ in an arc system.
	\end{enumerate}
	The last characterization is the one which we will use in diagrams.

	\begin{definition}
		\label{def:MultisectionDivides}
		A \emph{multisection with divides diagram} is a tuple $(\Sigma,d,C^1,\dots, C^n)$ such that
		\begin{itemize}
			\item $d$ cuts $\Sigma$ into two pieces $\Sigma^+$ and $\Sigma^-$ which are homeomorphic
			\item $C^i\cap \Sigma^\pm$ is an arc system for $i=1,\dots, n$.
			\item $(\Sigma,d,C^i,C^{i+1})$ is a contact Heegaard splitting diagram for the standard \emph{tight} contact structure on $\#_{k_i}S^1\times S^2$.
		\end{itemize}
	\end{definition}
	
	It was shown in~\cite{IslambouliStarkston} that every $4$-dimensional Weinstein domain admits a multisection with divides which can be represented by such a diagram.

	\section{Contact cut graph} \label{s:ccdef}
	
	In this section we define of the analogue of the cut graph in the contact handlebody setting.

	Given a closed oriented surface $ \Sigma_g $, a multicurve $ d $ is called a \emph{standard contact dividing set} if it separates $ \Sigma_g $ into two surfaces with boundary $\Sigma^+$ and $\Sigma^-$ which are orientation reversing homeomorphic to each other.	

	\begin{definition}
		Given a closed oriented surface $ \Sigma_g $ with standard contact dividing set $ d $ separating $\Sigma = \Sigma^-\cup_d \Sigma^+$, a \emph{contact cut system} is a cut system such that its intersection with each of $\Sigma^-$ and $\Sigma^+$ is an arc system.
	\end{definition}

	See Figure~\ref{fig:contact_cut_example} for an example. 

	\begin{figure}
		\centering
		\includegraphics[scale=0.2]{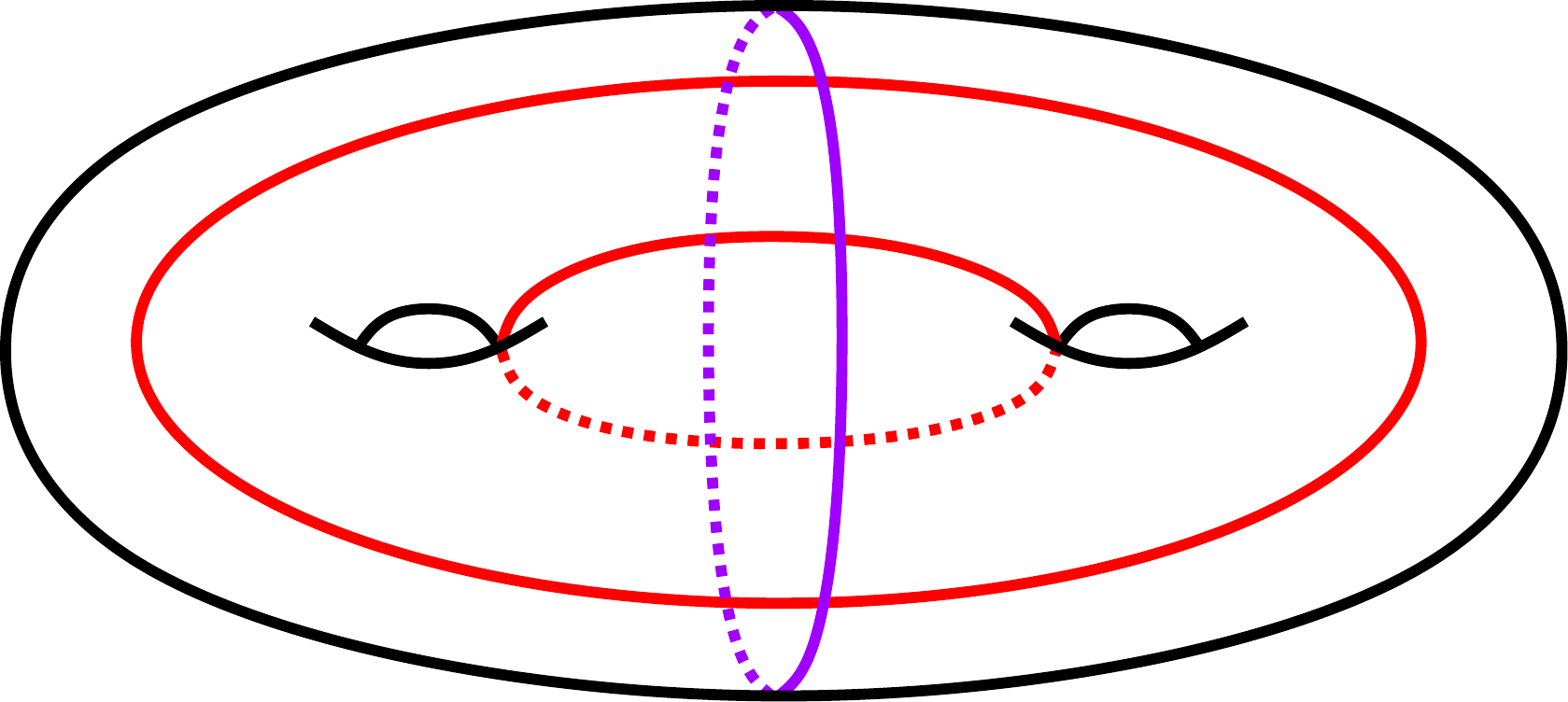}
		\caption{A contact cut system. The red curves form an complete arc system on either side of the purple dividing set.}
		\label{fig:contact_cut_example}
	\end{figure}

	In general, there are many different cut systems which define the same topological handlebody, related by a sequence of type $0$ moves (handleslides). We would like to understand which of these moves will preserve the contact condition with respect to a fixed dividing set $d$. One type of handleslide which clearly preserves the requirement of being a contact cut system comes from performing arc slides in both $\Sigma^-$ and $\Sigma^+$ simultaneously. This motivates the following definition.	
	\begin{definition}	
		A \emph{contact handleslide} or \emph{contact type $0$ move} is a handleslide of one curve within a contact cut system over a sequence of other curves such that for each slide, the ribbon along which the slide is a regular neighborhood of an arc contained in $d$.
	\end{definition}
	Observe that in this definition, the arc contained in $d$ provides the path $\gamma$ in $\partial \Sigma^\pm$ along which we are performing arc slides in $\Sigma^+$ and $\Sigma^-$ simultaneously.
	
	In the classical setting, two cut systems $C_1$ and $C_2$ for $\Sigma$ are related by a type $1$ move if they intersect in a single point. When this occurs, one can show (by looking in a neighborhood of the two curves that intersect) that there exists a curve $V\subset \Sigma$ such that $C_2 = \tau_V^\pm(C_1)$ where $\tau_V^\pm$ is a Dehn twist about $V$. In general, $V$ need not respect the dividing set $d$. To be geometrically compatible with the contact and symplectic structure, we ask $V$ to be disjoint from the dividing set.
	
	\begin{definition}
		There is an \emph{oriented contact type $1$ move} from contact cut system $C_1$ to contact cut system $C_2$ (for $(\Sigma_g,d)$), if there exists a curve $V\subset \Sigma_g\backslash d$, such that $C_2$ is obtained from $C_1$ by performing a single right-handed Dehn twist along $V$. If the Dehn twist is left- or right-handed, we say $C_1$ and $C_2$ are related by an (unoriented) contact type $1$ move.
	\end{definition}

	\begin{definition}		
		The \emph{contact cut graph}, $ CC(\Sigma_g,d) $, is a subgraph of $ C(\Sigma_g) $ whose vertices correspond to isotopy classes of contact cut systems for $ (\Sigma_g,d) $ and whose edges correspond to contact type $0$ and contact type $1$ moves, where type $1$ edges are oriented from $v_i$ to $v_{j}$ where $v_{j}$ is obtained from $v_i$ by a right-handed Dehn twist.
	\end{definition}

	Next, we will explore some preliminary properties of contact type $0$ and contact type $1$ moves. To do this, we first need a lemma about arc systems.

	\begin{lemma}\label{lem:arcslides}
		Let $F$ be a surface with boundary, and let $\mathcal{A}$ and $\mathcal{A}'$ be two arc systems for $F$. Then $\mathcal{A}$ and $\mathcal{A}'$ can be related by a sequence of arc slides.
	\end{lemma}

	\begin{proof}
		Note that $\mathcal{A}$ or $\mathcal{A}'$ can be interpreted as the collection of co-cores of the $1$--handles for a handle decomposition of the surface. By turning this handle decomposition upside down, $\mathcal{A}$ becomes the collection of cores of the $1$-handles. There is a collection of handleslides taking the cores of one handle decomposition to the cores of another handle decomposition of a surface with boundary. These handleslides for the relative handle decomposition are exactly arc slides. Thus, $\mathcal{A}$ and $\mathcal{A}'$ are related by arc slides.
	\end{proof}

	A priori, it could be possible to have different equivalence classes of contact handlebodies in the same smooth handlebody equivalence class. However, we show that this is not the case.
	
	\begin{proposition}\label{prop:type0}
		If two contact vertices $v, v' \in CC(\Sigma_g,d)$ are related through a smooth type $0$ move, then they are related by a contact type $0$ move.
	\end{proposition}
	
	\begin{proof}
		Let $v=(\alpha_1,\dots, \alpha_{g-1},\alpha_g)$ and $v'=(\alpha_1,\dots, \alpha_{g-1},\tilde \alpha_g)$. The dividing set $d$ cuts the surface $\Sigma$ into two homeomorphic surfaces with boundary, $\Sigma^-$ and $\Sigma^+$. Let $a_i^\pm = \alpha_i\cap \Sigma^\pm$ and $\tilde a_g^\pm = \tilde\alpha_g\cap \Sigma^\pm$. Since $(a_1^-,\dots,a_{g-1}^-, a_g^-)$ and $(a_1^-,\dots, a_{g-1}^-,\tilde a_g^-)$ are both complete cut systems of arcs for $\Sigma^-$, there exists a sequence of arc slides between them. Apply the corresponding contact handleslides to $\alpha_g$ to obtain the contact cut system $(\alpha_1,\dots, \alpha_{g-1},\overline\alpha_g)$. By assumption, $\tilde \alpha_g$ was disjoint from $\alpha_1,\dots, \alpha_g$. Note that performing contact handleslides cannot create new intersections because the sliding ribbon is disjoint from $\alpha_1,\dots, \alpha_g$. Therefore $\overline\alpha_g$ is also disjoint from $\alpha_1,\dots, \alpha_g$. We claim that $\overline \alpha_g$ is isotopic to $\tilde \alpha_g$. Indeed, by construction $\overline a_g^-:=\overline \alpha_g \cap \Sigma^-$ is isotopic to $\tilde a_g^-$, and after an isotopy we can assume that these coincide. Let $\overline a_g^+= \overline \alpha_g \cap \Sigma^+$. Then $(a_1^+,\dots, a_{g-1}^+,\tilde a_g^+)$ and $(a_1^+,\dots, a_{g-1}^+,\overline a_g^+)$ are both complete cut systems of arcs for $\Sigma^+$, and $\overline a_g^+$ is disjoint from $a_1^+,\dots, a_{g-1}^+,\tilde a_g^+$ away from its end points. Moreover, up to isotopy, the end points of $\overline a_g^+$ and $\tilde a_g^+$ agree because of the contact handleslides lining up the arc systems on $\Sigma^-$. There is a unique way up to isotopy to connect two end-points disjointly from $a_1^+,\dots, a_{g-1}^+,\tilde a_g^+$ (since these arcs cut $\Sigma^+$ into a disk), therefore $\overline a_g^+$ must be isotopic to $\tilde a_g^+$ through an isotopy fixing the end points.
	\end{proof}

	\begin{proposition}\label{prop:type1}
		If two contact vertices $v, v' \in C(\Sigma_g)$ are related by a smooth type 1 move, then after some sequence of contact handleslides, they are related by a contact type 1 move.
	\end{proposition}
	\begin{proof}
		Let us use the same notation as in Proposition~\ref{prop:type0}, except now $\alpha_g$ and $\tilde \alpha_g$ have a single transverse intersection.  Without loss of generality, we may assume that the intersection occurs on $\Sigma^+$. Thus, $(a_1^-,\dots, a_{g-1}^-,a_g^-)$ and $(a_1^-,\dots, a_{g-1}^-,\tilde a_g^-)$ are complete cut systems of arcs for $\Sigma^-$. By Lemma~\ref{lem:arcslides}, there is a sequence of arc slides relating them. We perform the corresponding contact handleslides to $\alpha_g$, and denote the resulting curve by $\overline\alpha_g$. Since contact handleslides do not change the number of intersection points, there is a unique intersection point of $\tilde \alpha_g$ with $\overline \alpha_g$. We now cut $\Sigma^+$ along $a_1^+,\dots, a_{g-1}^+$,  which results in an annulus. Each boundary component of the annulus alternates between the ``scars" left by cutting along $a_1^+,\dots, a_{g-1}^+$ and arcs which came from $\partial \Sigma^+$. Since $\overline a_g^+$ cuts the annulus into a disk, we can assume that it is the vertical arc as shown in Figure~\ref{fig:Annuli}. We have arranged that the end-points of $\tilde a_g^+$ are small push-offs of the end-points of $\overline a_g^+$. By assumption, $\overline a_g^+$ and $\tilde a_g^+$ intersect exactly once, thus they are related by a Dehn twist about a curve $c\subset \Sigma^+$, as in Figure~\ref{fig:Annuli}. Therefore, $v$ and $v'$ are related by potentially a contact type $0$ move followed by a contact type $1$ move.
		
		\begin{figure}
			\includegraphics[scale=0.2]{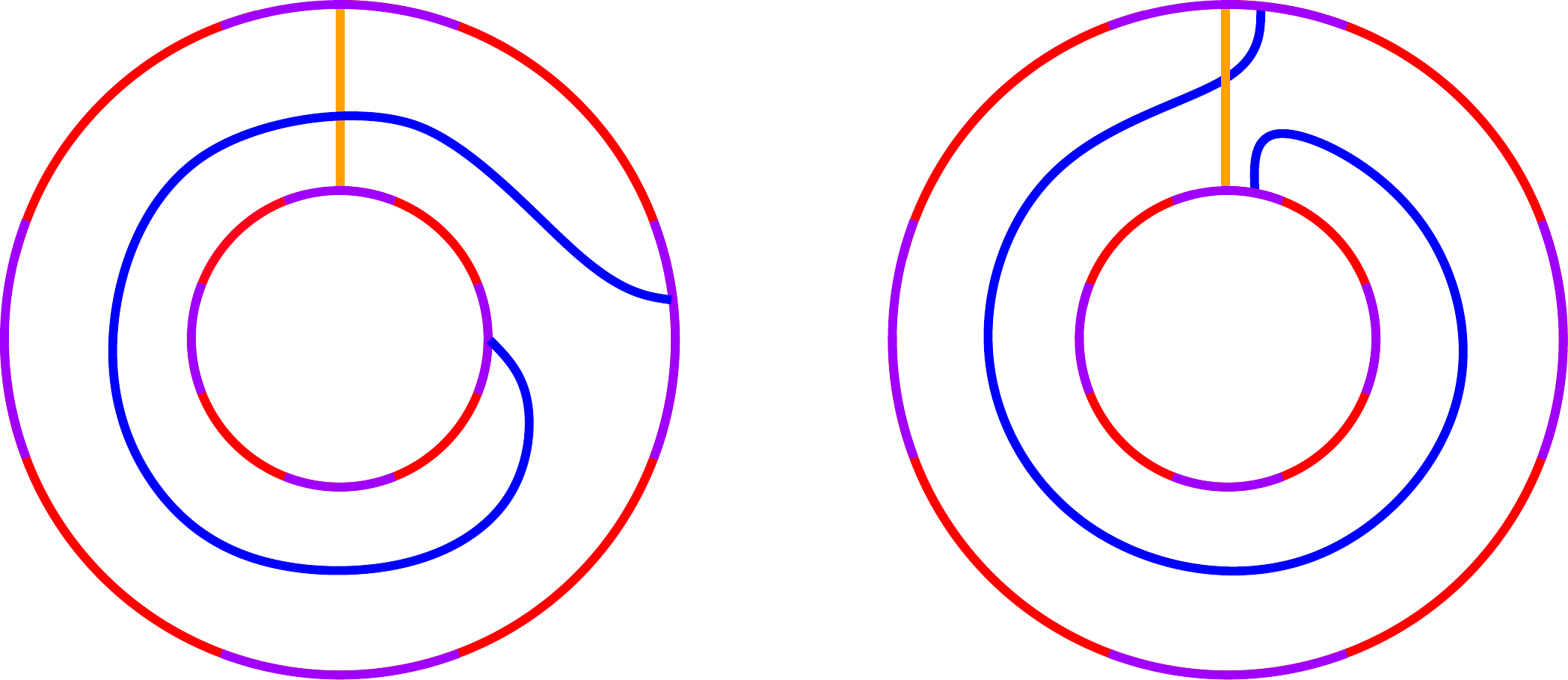}
			\caption{The two blue arcs are related through a sequence of arc slides, hence the contact vertices they represent are related by type 0 moves. The blue arc in the second picture is related to the orange arc by a Dehn twist, and therefore the contact vertices they represent are related by a type 1 move.}
			\label{fig:Annuli}
		\end{figure}
	\end{proof}

	\section{Connectedness of Contact Cut Graph}
	\label{s:Conn}

	A central feature of the cut graph is its connectedness, which allows one to deduce generators for any group acting transitively on the graph. In the course of providing a presentation for the mapping class group of a closed orientable surface, Hatcher and Thurston  \cite{HT80} showed that any two cut systems in the cut complex could be connected using only type-1 edges. A theorem of Reidemeister and Singer \cite{Rei33} \cite{Sin33} implies that any two cut systems defining the same handlebody can be connected by a sequence of only type-0 edges. This was later utilized by Wajnyrb \cite{Waj98} to provide a presentation for the mapping class group of a handlebody.
	  In the present work, we are studying a subgraph of the standard cut graph made up of only contact type vertices and contact type edges. While the larger cut graph is connected, the contact cut subgraph could in principle have multiple connected components. However, we will prove that the contact cut graph is in fact connected as well.
	
	\begin{theorem}\label{thm:connected}
		The (unoriented) contact cut graph $CC(\Sigma_g,d)$ is connected.
	\end{theorem}

	Observe that there is an intermediary graph between the smooth cut graph $C(\Sigma)$ and the contact cut graph $CC(\Sigma,d)$, where we include only the contact type vertices, but include all edges from $C(\Sigma)$ which connect pairs of contact vertices. We will denote this graph by $CC^*(\Sigma,d)$. A consequence of Propositions~\ref{prop:type0} and~\ref{prop:type1} is that if $CC^*(\Sigma,d)$ is connected, then $CC(\Sigma,d)$ is connected as well. We know that the full cut graph $C(\Sigma)$ is connected \cite{HT80}. However, this does not immediately imply that $CC^*(\Sigma,d)$ is connected because in principle there may be two contact vertices which are connected in $C(\Sigma)$ only by paths which pass through non-contact vertices. Thus, we need to make further arguments to prove that $CC^*(\Sigma,d)$ and $CC(\Sigma,d)$ are connected graphs. (Note that since they have the same set of vertices, and the edges of $CC(\Sigma,d)$ are a subset of that of $CC^*(\Sigma,d)$, proving that $CC(\Sigma,d)$ is a connected graph is a priori strictly stronger, but based on Propositions~\ref{prop:type0} and~\ref{prop:type1}, proving connectedness of $CC^*(\Sigma,d)$ or of $CC(\Sigma,d)$ are equivalent.)
	
	We need one more useful lemma, before we can prove the main theorem.
	
	Note that a sequence of arc slides between complete arc systems $\mathcal{A}$ and $\mathcal{A}'$ for a surface $F$ is specified by a sequence of intervals $\gamma_1,\dots, \gamma_n \subset \partial F$. This allows us to compare sequences of arc slides after a diffeomorphism on $F$ which acts as the identity on a neighborhood of $\partial F$.
	
	\begin{lemma}\label{lem:arcslidediffeo}
		Let $F$ be a surface with boundary, and let $\mathcal{A}$ and $\mathcal{A}'$ be complete arc systems for $F$ which are related by a sequence of arc slides specified by the intervals $\gamma_1,\dots, \gamma_n\subset \partial F$. Let $\phi:F\to F$ be a diffeomorphism of $F$ which is the identity on a neighborhood of $\partial F$. Then $\phi(\mathcal{A})$ and $\phi(\mathcal{A}')$ are related by the sequence of arcslides specified by $\gamma_1,\dots, \gamma_n$.
	\end{lemma}
	
	\begin{proof}
		It suffices to check this for a single arc slide along the interval $\gamma$. (Then repeatedly apply the result for a finite sequence of arc slides.)
		
		Denote by $a_1$ and $a_2$ the arcs in $\mathcal{A}$ with endpoints $a_1^+$ and $a_2^-$ on $\gamma$. The arc slide from $\mathcal{A}$ to $\mathcal{A}'$ is obtained by replacing $a_1$ by $a_1'$, the component of the boundary of a closed regular neighborhood $N$ of $a_1\cup \gamma \cup a_2$ whose end points are closest to $a_1^-$ and $a_2^+$. Since $\phi$ fixes a neighborhood of $\partial F$, $\phi$ fixes $\gamma$ (and the portion of $N$ which is a regular neighborhood of $\gamma$). Additionally, $\phi(a_1)$ and $\phi(a_2)$ are the arcs of $\phi(\mathcal{A})$ whose endpoints lie on $\gamma$. Therefore $\phi(N)$ is a regular neighborhood of $\phi(a_1)\cup \gamma \cup \phi(a_2)$. Therefore the arc slide of $\phi(\mathcal{A})$ along $\gamma$ is the result of replacing $\phi(a_1)$ with the component of $\partial(\phi(N))$ which is isotopic to a concatenation of $\phi(a_1)$, $\gamma$, and $\phi(a_2)$. Since $\phi(\partial N)=\partial \phi(N)$, and $\phi$ fixes a neighborhood of $\partial F$, this component of $\partial(\phi(N))$ is exactly $\phi(a_1')$. Thus, the result of the arc slide of $\phi(\mathcal{A})$ along $\gamma$ is precisely $\phi(\mathcal{A}')$ as claimed.
	\end{proof}

Now we are equipped to prove our main theorem that the contact cut graph is connected. The main idea is to extract the \emph{monodromy} of a pair of contact vertices.

	\begin{proof}[Proof of Theorem~\ref{thm:connected}]
		Let $v$ and $v'$ be vertices in the contact cut graph $CC(\Sigma,d)$. Denote the cut system corresponding to $v$ by $\alpha=\{\alpha_1,\dots, \alpha_g\}$, and the cut system corresponding to $v'$ by $\beta=\{\beta_1,\dots, \beta_g\}$. Together, $v$ and $v'$ define a contact Heegaard splitting of some $3$-manifold. This contact Heegaard splitting corresponds to an open book decomposition with page $\Sigma^-$. 
		
		To understand the monodromy of this open book decomposition, we look at the cut systems $\alpha$ and $\beta$. Note that $\alpha$ intersects each of $\Sigma^-$ and $\Sigma^+$ in an arc system. These arc systems meet along their endpoints in $d=-\partial \Sigma^-=\partial\Sigma^+$. Since $\Sigma^-$ and $\Sigma^+$ are diffeomorphic surfaces, and $\alpha^-$ and $\alpha^+$ are corresponding complete arc systems whose endpoints agree, there exists a diffeomorphism $\phi: \Sigma^+\to \Sigma^-$ fixing the boundary which sends $\alpha^+$ to $\alpha^-$. Thus, we can think of $\Sigma$ as the ``double'' of $\Sigma^-$ and $\alpha$ as the ``double'' of $\alpha^-$ ($\Sigma=\Sigma^-\cup_\partial \Sigma^-$ and $\alpha = \alpha^-\cup_\partial \alpha^-$). Next, consider the analogous arc systems $\beta^-$ and $\beta^+$ in $\Sigma^-$ and $\Sigma^+$. Since $\alpha^-$ and $\beta^-$ are both arc systems for $\Sigma^-$, there exists a sequence of arc slides relating them by Lemma~\ref{lem:arcslides}. Let $\tilde\beta$ be the cut system obtained from $\beta$ by performing the sequence of contact handleslides corresponding to this sequence of arc slides. Then $\tilde\beta^+:=\tilde\beta\cap\Sigma^+$ is an arc system for $\Sigma^+$ whose end-points agree with those of $\beta^-=\alpha^-$, and thus agree with the end-points of $\alpha^+$. Thus there exists a diffeomorphism $\psi:\Sigma^+\to \Sigma^+$ which takes $\tilde\beta^+$ to $\alpha^+$. The monodromy of the open book is then $m:=\phi\circ\psi\circ\phi^{-1}:\Sigma^-\to \Sigma^-$.
		
		The mapping class group of a surface (with boundary) is generated by Dehn twists about essential simple closed curves. Thus $m$ can be written as a product of (left- and right-handed) Dehn twists.
		$$m=\tau_{c_n}^\pm\circ \dots \circ \tau_{c_1}^\pm.$$
		
		Let $\mathcal{A}^i$ be an arc system for $\Sigma^\pm$ (where we identify $\Sigma^+$ with $\Sigma^-$ via $\phi$), such that $c_i$ intersects $\tau_{c_{i-1}}^\pm\circ \dots \circ \tau_{c_1}^\pm(\mathcal{A}^i)$ in exactly one point (so one of the arcs intersects $c_i$ at one point, and the other arcs are disjoint from $c_i$). Let $v_i$ be the contact cut system for $(\Sigma,d)$ which is given by $\mathcal{A}^i$ on the $\Sigma^-$ side, and by $\tau_{c_{i-1}}^\pm\circ \dots \circ \tau_{c_1}^\pm(\mathcal{A}^i)$ on the $\Sigma^+$ side. Let $v_i'$ be the contact cut system for $(\Sigma,d)$ which is given by $\mathcal{A}^i$ on the $\Sigma^-$ side, and by $\tau_{c_i}^\pm\circ\tau_{c_{i-1}}^\pm\circ \dots \circ \tau_{c_1}^\pm(\mathcal{A}^i)$ on the $\Sigma^+$ side. 
		
		We make the following observations/claims:
		\begin{enumerate}
			\item \label{item:t1} For $i=1,\dots, n$, $v_i$ and $v_i'$ are related by a contact type 1 move.
			\item \label{item:t0a} $v$ and $v_1$ are related by contact type $0$ moves.
			\item \label{item:t0b} For $i=1,\dots, n-1$, $v_i'$ and $v_{i+1}$ are related by contact type $0$ moves.
			\item \label{item:t0c} $v_n'$ and $v'$ are related by contact type $0$ moves.
		\end{enumerate}
	
		Item (\ref{item:t1}) follows from the definitions of $v_i$ and $v_i'$.
		Items (\ref{item:t0a}) and (\ref{item:t0b}) follow from the definitions of $v$, $v_i$, $v_i'$, and $v'$ together with Lemma~\ref{lem:arcslides} and Lemma~\ref{lem:arcslidediffeo}. For item (\ref{item:t0c}), note that the cut system representing $v_n'$ is related to the cut system $\tilde{\beta}$ by contact type $0$ moves by Lemma~\ref{lem:arcslides} and Lemma~\ref{lem:arcslidediffeo}, and $\tilde{\beta}$ is related to cut system $\beta$ by contact type $0$ moves by definition of $\tilde{\beta}$. Since $\beta$ represents vertex $v'$, claim (\ref{item:t0c}) follows.
		
		Thus we obtain a path $(v,v_1,v_1',\dots, v_n,v_n',v')$ from $v$ to $v'$ in $CC^*(\Sigma,d)$ of contact type 0 moves and contact type 1 moves.
	\end{proof}

\section{Defining the Weinstein $\cL$-invariant and Lefschetz fibrations}
\label{s:LF}
%
%
%
	The primary goal of this section is to define our new complexity measure for Weinstein domains, the Weinstein $\cL$-invariant. This requires developing a process to obtain a Weinstein domain from an oriented path in a contact cut graph, and showing that every Weinstein domain can be obtained in this way. There are two intermediary stages in this process: multisections with divides and their diagrams and Lefschetz fibrations. Building off of these connections, in Section~\ref{ss:folded}, we will give a correspondence for smooth 4-manifolds via \emph{achiral} Lefschetz fibrations.
		
	Kirby and Thompson defined the $\cL$-invariant for closed smooth $4$-manifolds in terms of trisections and distances in the cut-graph \cite{KT18}. In the years following their work, the $\mathcal{L}$-invariant has been extended to various settings including for smooth manifolds with boundary \cite{CIMT21}, knotted surfaces in $B^4$ \cite{BCTT22}, and knotted surfaces in closed manifolds \cite{APTZ21}. These invariants utilize relative trisections, relative bridge trisection, and bridge trisections, respectively, which give rise to paths in an appropriately defined graph from which a length can be extracted.
	
	In Section~\ref{ss:LF}, from an oriented path $P$ in $CC(\Sigma,d)$ we will define a multisection with divides $\mathfrak{M}(P)$ supporting a Weinstein domain $W(P)$. Using this, we can define our measure of complexity for Weinstein domains.
	
	For any path $P$ in $CC(\Sigma,d)$, let $N_0(P)$ denote the number of type $0$ moves in the path.
	
	\begin{definition}
		Let $\mathfrak{M}$ be a multisection with divides with core $(\Sigma,d)$. We define the \emph{Weinstein $\cL$-invariant of $\mathfrak{M}$} to be
		$$\cL_{(\Sigma,d)}(\mathfrak{M}) := \min \{ N_0(P) \mid P \textrm{ is a path in }CC(\Sigma,d), \textrm{ such that } \mathfrak{M}(P) = \mathfrak{M} \}.$$
		If there is no path $P$ with $\mathfrak{M}(P)=\mathfrak{M}$ we set $\cL_{(\Sigma,d)}(\mathfrak{M})=\infty$.	
		
		Let $(W,\omega, V)$ be a Weinstein domain. We define the \emph{Weinstein $L$-invariant of $(W,\omega, V)$} to be
		$$\cL(W,\omega,V) := \min \{\cL_{(\Sigma,d)}(\mathfrak{M}) \mid \mathfrak{M} \textrm{ is a multisection with divides supporting } (W,\omega, V) \}.$$
		We also define the \emph{genus $ g$ $ L$-invariant of $ (W,\omega,V)$} to be \begin{align*}
			\cL_g(W,\omega,V):=\min \{\cL_{(\Sigma,d)}(\mathfrak{M}) \mid  & \mathfrak{M} \textrm{ is a multisection with divides}\\ &\textrm{supporting } (W,\omega, V), g(\Sigma)=g 	
		\}.\end{align*}
	\end{definition}
	
	\begin{remark}
		\begin{enumerate}
			\item It is possible to find a multisection with divides $\mathfrak{M}$ with $\cL_{(\Sigma,d)}(\mathfrak{M})=\infty$. This is because there are contact Heegaard splittings of $\#_k S^1\times S^2$ corresponding to open book decompositions whose monodromy has no factorization into right-handed Dehn twists~\cite{Wand,BEVHM,BW}. This means that the sector cannot be symplectically represented by an oriented path of contact type 1 moves.
			\item On the other hand, we will show in Theorem~\ref{thm:WeinsteinHasPath} that every Weinstein domain $(W,\omega, V)$ has \emph{finite} $\cL(W,\omega,V)$, by proving that every Weinstein domain is obtained from an oriented path in $CC(\Sigma,d)$ for some $(\Sigma,d)$ by our procedure.
		\end{enumerate}
	\end{remark}
	

\subsection{Lefschetz fibrations, Multisections with divides, and Paths in $CC(\Sigma,d)$} \label{ss:LF}

A Lefschetz fibration on a $4$-manifold (possibly with boundary) is a map $\pi:X\to S$ where $S$ is a surface such that for all critical points of $\pi$ there exist (orientation preserving) local complex coordinates $(z_1,z_2)$ on $X$ such that $\pi(z_1,z_2) = z_1^2+z_2^2$. When $X$ has boundary and $S$ is a disk, the Lefschetz fibration induces an open book decomposition on $\partial X$. Lefschetz fibrations and pencils have close connections to symplectic structures through work of Donaldson~\cite{Don99} and Gompf~\cite{GompfStipsicz,GompfLF}, which showed every symplectic manifold has a Lefschetz pencil and every Lefschetz pencil has a symplectic structure. For Weinstein domains, the correspondence is with ``allowable'' Lefschetz fibrations over a disk. Allowable means the vanishing cycles represent a non-zero class in $H_1(F)$ where $F$ is the fiber.

\begin{theorem}[\cite{LoiPie01,AkbulutOzbagci}] \label{thm:LFWeinsteincorrespondence}
	Every $4$-dimensional Weinstein domain $(W,\omega,V)$ admits a compatible allowable Lefschetz fibration $\pi:W\to D^2$ over the disk. Conversely, every 4-manifold $W$ admitting an allowable Lefschetz fibration $\pi:W\to D^2$ admits a compatible Weinstein structure which is unique up to Weinstein homotopy.
\end{theorem}

\begin{remark}
While symplectic geometry comes with natural orientations, and thus requires the local model coordinates for the Lefschetz critical points to be orientation preserving, in the smooth category one can relax the orientation preserving condition. The resulting maps $\pi:X\to S$ where critical points have the same models but are allowed to use orientation reversing coordinate charts are called \emph{achiral Lefschetz fibrations}.
\end{remark}

Given a Lefschetz fibration over $D^2$, with an order of the critical values, there is a natural multisection with divides associated to it. We build this by choosing two points $x_0$ and $x_1$ on the boundary of $D^2$, and a sequence of paths $\gamma_i$ in the disk connecting these two points such that the paths cut the disk into bigons where each bigon contains a unique critical value ordered left to right as in Figure~\ref{fig:LF}. Let $F=F_0=\pi^{-1}(x_0)$ 
and $F_1=\pi^{-1}(x_1)$. Then $W$ is supported by a multisection with divides whose sectors $X_1,\dots, X_n$ are the preimages of the bigon regions and the boundary contact handlebodies $H_0,H_1,\dots, H_n$ are the preimages of the paths connecting $x_0$ with $x_1$ as shown in Figure~\ref{fig:LF}. This was shown to be a multisection with divides in~\cite[Theorem 4.2]{IslambouliStarkston}. We will say that this is the \emph{multisection with divides compatible with the Lefschetz fibration}.

\begin{figure}
	\centering
	\includegraphics[scale=0.2]{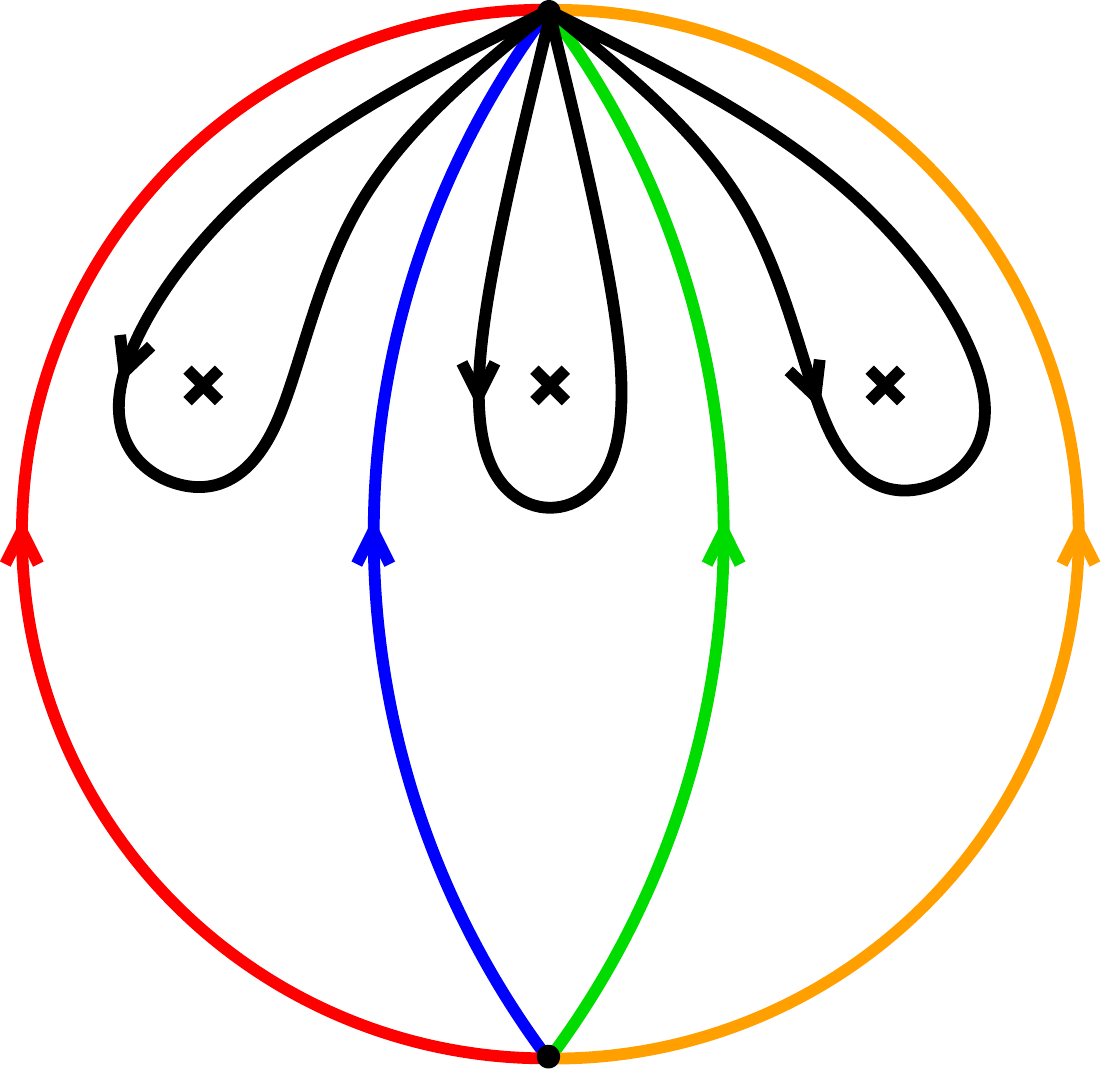}
	\caption{Building a multisection with divides compatible with a Lefschetz fibration.}
	\label{fig:LF}
\end{figure}

\begin{remark}
	If one does not care about the symplectic structure, the same construction produces a multisection with boundary from an achiral Lefschetz fibration. The cut systems for the handlebodies are still compatible with the dividing set (contact cut systems), but the contact structures on $\#_k S^1\times S^2=H_i\cup H_{i+1}$ will not necessarily be tight, and the 4-dimensional sectors will not generally support compatibly gluing symplectic structures. We will call such a multisection an \emph{achiral multisection with divides}.
\end{remark}

To pass between multisections with divides and paths in the contact cut graph, we need to use \emph{diagrams} for the multisections with divides. The following lemma provides a description of a collection of multisection with divides {diagrams} representing the multisection with divides compatible with a Lefschetz fibration.

\begin{lemma}\label{l:LFMDdiagrams}
	Consider a Lefschetz fibration over a disk with fiber $F$ and vanishing cycles $V_1,\dots, V_n$. Then any of the following is a {diagram} for the multisections with divides compatible with this Lefschetz fibration:
	
	Let $\Sigma = F\cup_{\partial F} \overline{F}$ with dividing set $d=\partial F$. Choose an arc system $\mathcal{A}$ for $F$. Let $C^0=\mathcal{A}\cup \overline{\mathcal{A}}$. For $i=1,\dots, n$, let $C^i = \tau_{T_i}C^{i-1}$ where $T_i$ is one of the following two options:
	\begin{enumerate}
		\item $T_i$ lies in $F$ and $T_i=V_i$ or
		\item $T_i$ lies in $\overline{F}$  and $T_i=(\tau_{V_{i-1}}\circ \cdots \tau_{V_1})^{-1}(V_i)=\tau_{V_1}^{-1}\circ \cdots \circ \tau_{V_{i-1}}^{-1}(V_i)$ with respect to the $F$ orientation. 
		
		Equivalently,
		$$T_i = \tau_{\tau_{V_1}\circ \cdots \circ \tau_{V_{i-2}}(V_{i-1})}\circ \tau_{\tau_{V_1}\circ \cdots \circ \tau_{V_{i-3}}(V_{i-2})} \circ \cdots \circ \tau_{\tau_{V_1}(V_2)}\circ \tau_{V_1}(V_i)$$
		with respect to the $\overline{F}$ orientation.
	\end{enumerate}
\end{lemma}

\begin{proof}
	We let $F_0$ be the base fiber over the point in the boundary of the disk where the vanishing cycles are measured as the monodromy around a loop based at $\pi(F_0)$ going counterclockwise around the $i^{th}$ critical value as in Figure~\ref{fig:LF}, such that the concatentation of the ordered loops is homotopic to the boundary of the disk.
	
	Let $F=F_0$. Let $\gamma_i$ be the path such that $H_i=\pi^{-1}(\gamma_i)$. Then there is a diffeomorphism $\pi^{-1}(\gamma_i)\cong F\times [0,1]$ such that $F_0$ is identified with $F$ by the identity. Restricting this diffeomorphism to $F_1$ gives a diffeomorphism $\psi_i:F_1\to F$ (by identifying $F\times\{1\}$ with $F$ by the projection), namely parallel transport along $\gamma_i$.
	
	Choosing any arc system $\mathcal{A}$ for $F$, $\mathcal{A}\times[0,1]$ is a system of compressing disks for $F\times[0,1]$, which is a contact cut system for $d=\partial F\times\{\frac{1}{2}\}$. Conversely, any contact cut system for $(F\times[0,1],d)$ has the form $\mathcal{A}\times[0,1]$ for some arc system $\mathcal{A}$. 
	
	By definition, for any contact cut system $C^i$ for $(\Sigma=F_0\cup_\partial F_1,d=\partial F_0)$, $\mathcal{A}^i := C^i\cap F_0$ is an arc system for $F_0$ and $\mathcal{B}^i = C^i\cap F_1$ is an arc system for $F_1$. Identifying $H_i$ with $F\times[0,1]$, we see that the contact cut system defines $H_i$ if and only if $\psi_i(\mathcal{B}^i) = \mathcal{A}^i$.
	
	The monodromy about a loop based at $x_0$, contained in $X_i$, which goes once counter-clockwise around the critical value is $\tau_{V_i}$. (If the Lefschetz fibration were achiral this could be $\tau_{V_i}^{-1}$.) Since the concatenation of the path on the left side of the $i^{th}$ bigon with such a loop is homotopic to the path on the right side of the $i^{th}$ bigon, we obtain
	\begin{equation}\label{eqn:monodromy}
		\psi_i = \tau_{V_i}\circ \psi_{i-1} = \tau_{V_i} \circ \cdots \circ \tau_{V_1} \circ \psi_0
	\end{equation}
	We will use this formula to verify that the multisection with divides diagrams in the statement of the lemma have the property that $\psi_i(\mathcal{B}^i)=\mathcal{A}^i$ for $i=1,\dots, n$, for some identification of $\Sigma$ with $F_0\cup_\partial F_1$, and thus these diagrams represent the multisection with divides compatible with the Lefschetz fibration.
	
	We use the identification $id\cup \psi_0: F_0\cup_\partial F_1 \to \Sigma = F\cup_\partial \overline{F}$ to identify the diagrams in the statement of the lemma with the core surface of the multisection with divides compatible with the Lefschetz fibration. Note that the orientation on $F_1$ disagrees with the orientation of $\overline{F}\subset \Sigma$ under this identification. 
	
	Since $C^0 = \mathcal{A}\cup \overline{\mathcal{A}}$, $\mathcal{A}^0 = \mathcal{A}$ and $\mathcal{B}^0\subset F_1$ is given by $\psi_0^{-1}(\mathcal{A})$. Thus, by definition $\psi_0(\mathcal{B}_0) = \mathcal{A}_0$. This establishes a base case for induction on $i$. 
	
	We inductively assume that $C^{i-1}$ is a cut system for $H^{i-1}$. Thus $\psi_{i-1}(\mathcal{B}_{i-1}) = \mathcal{A}_{i-1}$. From the statement of the lemma $C^i = \tau_{T_i}(C^{i-1})$ for some $T_i$ disjoint from the dividing set.

	\textbf{Case 1:} $T_i$ lies on the $F$ side and $T_i=V_i$. 
	
	Then
	$$\mathcal{B}^i=\mathcal{B}^{i-1}=\psi_{i-1}^{-1}(\mathcal{A}^{i-1})$$
	and 
	$$\mathcal{A}^i = \tau_{V_i}(\mathcal{A}^{i-1}).$$
	Thus, using equation~\ref{eqn:monodromy}, we obtain:
	$$\mathcal{A}^i = \tau_{V_i}\circ \psi_{i-1}(\mathcal{B}^{i-1}) = \tau_{V_i}\circ \psi_{i-1}(\mathcal{B}^i) = \psi_i(\mathcal{B}^i).$$
	So $C^i$ defines the contact handlebody $H_i$.
	
	\textbf{Case 2:} $T_i$ lies on the $\overline{F}$ side and $T_i=(\tau_{V_{i-1}}\circ \cdots \tau_{V_1})^{-1}(V_i)$ (where the Dehn twists specified here are considered right-handed with respect to the $F$ orientation).
	
	In this case,
	$$\mathcal{A}^i= \mathcal{A}^{i-1}$$
	and
	$$\mathcal{B}^i = \tau_{\psi_0^{-1}(T_i)}^{-1}(\mathcal{B}^{i-1})$$
	
	
	Since $T_i = (\tau_{V_{i-1}}\circ \cdots \tau_{V_1})^{-1}(V_i)$, we have that
	$$V_i = (\tau_{V_{i-1}}\circ \cdots \tau_{V_1})\circ\psi_0 (\psi_0^{-1}(T_i)) = \psi_{i-1}(\psi_0^{-1}(T_i))$$
	by formula~\ref{eqn:monodromy}. Therefore
	$$\psi_i(\mathcal{B}^i) = \tau_{V_i}\circ \psi_{i-1}(\mathcal{B}^i) = \tau_{V_i}\circ \psi_{i-1} \circ \tau_{\psi_0^{-1}(T_i)}^{-1}(\mathcal{B}^{i-1}) = \tau_{V_i}\circ \psi_{i-1} \circ \tau_{\psi_0^{-1}(T_i)}^{-1} \circ \psi_{i-1}^{-1}(\mathcal{A}^{i-1}) $$
	$$ = \tau_{V_i}\circ \tau_{\psi_{i-1}\circ \psi_0^{-1}(T_i)}^{-1} (\mathcal{A}^{i})= \tau_{V_i}\circ \tau_{V_i}^{-1}(\mathcal{A}^i) = \mathcal{A}^i.$$
	So $C^i$ defines the contact handlebody $H^i$ in this case as well.
		
	
	If $T_i\subset \overline{F}$ satisfies 
	$$T_i=(\tau_{V_{i-1}}\circ \cdots \tau_{V_1})^{-1}(V_i)=\tau_{V_1}^{-1}\circ \cdots \circ \tau_{V_{i-1}}^{-1}(V_i)$$
	with respect to the $F$ orientation, then
	$$T_i = \tau_{V_1}\circ \cdots \circ \tau_{V_{i-1}}(V_i)$$
	with respect to the $\overline{F}$ orientation. 
	
	To complete the last part of the lemma, we will show that
	$$\tau_{V_1}\circ \cdots \circ \tau_{V_{i-1}} = \tau_{\tau_{V_1}\circ \cdots \circ \tau_{V_{i-2}}(V_{i-1})}\circ \tau_{\tau_{V_1}\circ \cdots \circ \tau_{V_{i-3}}(V_{i-2})} \circ \cdots \circ \tau_{\tau_{V_1}(V_2)}\circ \tau_{V_1}.$$
	
	
	This is implied by the conjugation formula for Dehn twists as follows.
	$$\tau_{V_1}\circ \cdots \circ \tau_{V_{i-1}} =  \tau_{V_1}\circ \tau_{V_2}\circ  \cdots \circ \tau_{V_{i-1}}\circ \tau_{V_1}^{-1}\circ \tau_{V_1}=$$
	$$(\tau_{V_1}\circ \cdots \circ \tau_{V_{i-2}} \circ \tau_{V_{i-1}} \circ \tau_{V_{i-2}}^{-1} \circ \cdots \circ \tau_{V_1}^{-1}) \circ (\tau_{V_1} \circ \cdots \circ \tau_{V_{i-3}}\circ \tau_{V_{i-2}}\circ \tau_{V_{i-3}}^{-1}\circ \cdots \circ \tau_{V_1}^{-1}) \circ \cdots \circ (\tau_{V_1}\circ \tau_{V_2}\circ \tau_{V_1}^{-1}) \circ  \tau_{V_1} $$
	$$=\tau_{\tau_{V_1}\circ \cdots \circ \tau_{V_{i-2}}(V_{i-1})}\circ \tau_{\tau_{V_1}\circ \cdots \circ \tau_{V_{i-3}}(V_{i-2})} \circ \cdots \circ \tau_{\tau_{V_1}^{-1}(V_2)}\circ \tau_{V_1}.$$
	
\end{proof}

\begin{remark}\label{rk:achiral}
	If some of the Dehn twists relating $\mathcal{C}^{i-1}$ to $\mathcal{C}^i$ are left-handed, then we can consider the achiral multisection with divides compatible with the achiral Lefschetz fibration. It will be supported by analogous diagrams following the same proof, replacing $\tau_A$ with $\tau_A^{-1}$ for the negatively oriented Lefschetz singularities.
\end{remark}

Now we outline our key procedure.

\begin{procedure} \label{mainprocedure}
	Given an oriented path $P$ in $CC(\Sigma,d)$, we obtain
	\begin{enumerate}
		\item \label{i:PtoDiag} \textbf{a multisection with divides} $\mathfrak{M}(P)$ with a particular \textbf{diagram}, $\mathfrak{MD}(P)$, by interpreting the vertices of $P$ as contact handlebody diagrams,
		\item \label{i:PtoLF} \textbf{a Lefschetz fibration}, $LF(P)$, (unique up to Lefschetz fibration isomorphism) such that the multisection with divides compatible with $LF(P)$ can be represented by the diagram $\mathfrak{MD}(P)$, and
		\item \label{i:LFtoW} \textbf{a Weinstein domain}, $W(P)$, (unique up to Weinstein homotopy) supported by the Lefschetz fibration $L(P)$.
	\end{enumerate}
\end{procedure}

We obtain (\ref{i:LFtoW}) from (\ref{i:PtoLF}) by Theorem~\ref{thm:LFWeinsteincorrespondence}. Step (\ref{i:PtoLF}) will be justified in Theorem~\ref{thm:PtoLF} below. To justify Step (\ref{i:PtoDiag}), we want to see how the vertices of an oriented path $P$ in $CC(\Sigma,d)$ actually give rise to a multisection with divides diagram. We need to check that adjacent pairs of vertices represent a contact Heegaard splitting of $\#_k S^1\times S^2$ with the \emph{tight} contact structure. When the vertices are related by a contact type $0$ move, the two underlying contact handlebodies are the same so no new sector is needed. 
When the vertices are related by an oriented contact type $1$ move, the contact Heegaard splitting supports an open book decomposition whose monodromy is a single right handed Dehn twist about a curve. Thus the supported contact structure is Weinstein fillable so the contact structure is tight. (That vertices related by a (general) type 1 move give a Heegaard splitting of $\#_{g-1} S^1\times S^2$ is based on classical results of Waldhausen~\cite{Wald68} and Haken~\cite{Hak68}, see~\cite[Section 2.1]{Isl21}.) This verifies that $\mathfrak{MD}(P)$ is a multisection with divides diagram. By~\cite[Proposition 3.4]{IslambouliStarkston}, this diagram represents a multisection with divides $\mathfrak{M}(P)$ on a Weinstein domain $W(P)$ (note that the definition of standard Weinstein cobordant of~\cite[Proposition 3.4]{IslambouliStarkston} is precisely a contact type $1$ move).
 
\begin{theorem}\label{thm:PtoLF}
	Let $P$ be an oriented (resp. unoriented) path in $CC(\Sigma, d)$. There exists an (achiral) allowable Lefschetz fibration $LF(P)$ with fiber $\Sigma^+$ such that $\mathfrak{MD}(P)$ is a diagram for the (achiral) multisection with divides compatible with $LF(P)$.
\end{theorem}

\begin{proof}
	Given a path $P$ in the contact cut graph, vertices related by contact type $0$ moves correspond to the same contact handlebody, whereas contact type $1$ moves change the underlying contact handlebody. Let $H_0,H_1,\dots, H_n$ denote the sequence of contact handlebodies induced by the path $P$ by collapsing subpaths connected only by contact type $0$ moves. We will construct an (achiral) {allowable} Lefschetz fibration, such that the compatible (achiral) multisection with divides has contact handlebodies $H_0,H_1,\dots, H_n$. To do this, it suffices for each contact type $1$ move connecting $(H_{i-1},\xi_{i-1})$ to $(H_i,\xi_i)$, to identify a vanishing cycle $V_i$ {which is homologically essential}, such that the Lefschetz fibration over a disk with a single (achiral) Lefschetz singularity with vanishing cycle $V_i$ has induced contact Heegaard splitting on the boundary $(H_{i-1},\xi_{i-1})\cup_{\Sigma} (H_i,\xi_i)$. For this, we know that the contact cut systems $C_{i-1}$ and $C_i$ are related by a Dehn twist about a curve $T_i$ disjoint from $d$. If $T_i$ is contained in $\Sigma^+$, we let $V_i=T_i$, and if $T_i$ is contained in $\Sigma^-$ we let $V_i=\tau_{V_{i-1}}^\pm\circ \cdots \circ \tau_{V_1}^\pm(T_i)$ (here the signs depends on whether the $j^{th}$ contact type $1$ move represented a positive or negative Dehn twist about $T_j$ for $1\leq j \leq i-1$). Observe that $V_i$ is homologically essential because it is a dual curve in some arc system.
	By Lemma~\ref{l:LFMDdiagrams}, this path gives a diagram representing the (achiral) multisection with divides compatible with the (achiral) Lefschetz fibration.
\end{proof}


%
%

Finally, we show that every Weinstein domain arises as $W(P)$ for some path $P$.
This path is not necessarily unique.

\begin{theorem} \label{thm:WeinsteinHasPath}
	For any Weinstein domain $(W,\omega, V)$ (up to Weinstein homotopy), there exists a pair $(\Sigma,d)$ and path $P\in CC(\Sigma,d)$ such that $(W,\omega, V)=W(P)$.
\end{theorem}

\begin{proof}
	By Theorem~\ref{thm:LFWeinsteincorrespondence}, every Weinstein domain is supported by some allowable Lefschetz fibration with fiber $F$. Let $(\Sigma,d)= (F\cup_{\partial F} \overline{F},\partial F)$.  By the more general Theorem~\ref{thm:AlfsArePaths} below, there is a path $P\in CC(\Sigma,d)$ such that $LF(P)$ is this Lefschetz fibration.
\end{proof}

\begin{theorem}
	\label{thm:AlfsArePaths}
	For any allowable achiral Lefschetz fibration over $D^2$ with fiber $F$ there exists a (generally non-unique) path $P$ in $CC(\Sigma_F, d)$ such that $LF(P)$ is the given achiral Lefschetz fibration. (Here $(\Sigma_F,d)= (F\cup_{\partial F} \overline{F},\partial F)$.)
\end{theorem}

\begin{proof}
	First we obtain a multisection diagram compatible with the achiral Lefschetz fibration from Lemma~\ref{l:LFMDdiagrams}. We have choices for this diagram, but since we are not claiming uniqueness, we can make any choice. For simplicity, let's choose the diagram where all the $T_i$ occur on the $F$ side. For each $i=0,\dots, n$, we have a cut system $C^i$ corresponding to a vertex in the contact cut graph $CC(\Sigma_F,d)$. We claim that for each $i=1,\dots, n$, $C^{i-1}$ is related to $C^i$ by a sequence of contact type $0$ moves and a single contact type $1$ move. To see this, recall that $C^i$ is obtained from $C^{i-1}$ by performing a Dehn twist about $T_i\subset F$. If $T_i$ intersects $C^{i-1}$ in a single point, then this is precisely a contact type $1$ move. In general though, $T_i$ may intersect $C^{i-1}$ in more than one point. However, since $T_i$ is essential in $F$, there exists an arc system $\mathcal{A}$ for $F$ which intersects $T_i$ transversally in a single point. The arc system $C^{i-1}\cap F$ is related to $\mathcal{A}$ by a sequence of arc slides. If we perform the contact type $0$ moves corresponding to these arc slides to $C^{i-1}$, we will obtain a new contact cut system $\widetilde{C}^{i-1}$ which defines the same contact handlebody as $C^{i-1}$, but intersects $T_i$ transversally in a single point. Note that since $C^{i-1}$ and $C^i$ are related by a diffeomorphism $\phi$ of $F$ supported away from the boundary, it makes sense to talk about performing the corresponding sequence of arc slides and contact type $0$ moves on $C^i$ to obtain $\widetilde{C}^i$. By Lemma~\ref{lem:arcslidediffeo}, $\widetilde{C}^{i-1}$ and $\widetilde{C}^i$ are related by the Dehn twist about $T_i$. Since $T_i$ intersects $\widetilde{C}^{i-1}$ in a single point and $T_i$ is disjoint from $d$, $\widetilde{C}^{i}$, there is a contact type $1$ move relating $\widetilde{C}^{i-1}$ with $\widetilde{C}^i$. Thus, we obtain a path of contact type $0$ moves from $C^{i-1}$ to $\widetilde{C}^{i-1}$, a contact type $1$ move to $\widetilde{C}^i$, and another sequence of contact type $0$ moves to $C^i$. Concatenating these paths for $i=1,\dots, n$ gives a path in $CC(\Sigma_F,d)$ supporting the multisection with divides compatible with the achiral Lefschetz fibration on $X$.
\end{proof}

\subsection{Folded Lefschetz fibrations} \label{ss:folded}

We recall a definition.

\begin{definition}
Let $X$ be a closed smooth 4-manifold, and $N$ be a regular neighborhood of a disjoint collection of embedded circles in $X$. Let $S^2$ be decomposed as hemispheres $D^\pm$ meeting along the equator $C$. A smooth map $f: X-N \to S^2$ is called a \textbf{folded Lefschetz fibration} if $f$ restricts to a (positive) Lefschetz fibration on $D^+$, a negative Lefschetz fibration on $D^-$ and an open book decomposition on $C$. 
\end{definition}

\begin{theorem}\label{loopssmooth4viaALF}
For every closed, oriented smooth 4-manifold $X$, there exists a regular neighborhood $N$ of disjointly embedded circles such that $X\setminus N$ can be expressed as $X_C(L)$ for some unoriented loop $L$ in $CC(\Sigma, d)$.
\end{theorem}

\begin{proof}

Let $X$ be an arbitrary closed smooth 4-manifold. By Proposition 6.6 of \cite{Baykur06}, $X-N$ admits a folded Lefschetz fibration $f: X-N \to S^2$. Let $n$ and $s$ be the north and south poles respectively and let $\Sigma^-$ be $f^{-1}(s)$ and $\Sigma^+$ be $f^{-1}(n)$. Take the (latitude, longitude) coordinates $(\theta, \psi)$ on $S^2$. For $0 \leq i \leq 2 \pi$, we let $l_i$ be the line $l_{i} = \{ (\theta, \psi) \in S^2 | \theta = i, 0 \leq \psi \leq \pi$ with endpoints $n$ and $s$. Let $S_{[a,b]}$ be the sector of $S^2$ given by $S_{[a,b]} = { l_i |a \leq i \leq b}$.  By a homotopy through Lefschetz fibrations, we may assume that the (positive and negative) Lefschetz singularities are contained on the equator of $S^2$ and that no two critical points take the same critical value. Let $\theta_1, \dots , \theta_n$ be the angles where $l_{\theta_i}$ passes through a critical point and let $\epsilon \in \mathbb{R}$ be sufficiently small so that $S_{[\theta_i - \epsilon, \theta_i+ \epsilon]}$ contains only one Lefschetz singularity.  

Note that for a regular line, $l_r$, not passing through a critical point,  $f^{-1} (l_r)$ is a handlebody. We denote by $H_i$ the handlebodies corresponding to $f^{-1}(l_{\theta_i + \epsilon})$. Let $\Sigma = \Sigma^+ \cup_d \Sigma^-$ where $d = \partial \Sigma^+= \partial \Sigma^-$. As in the construction of a multisection with divides compatible with a Lefschetz fibration, this decomposition gives an achiral multisection with divides where $H_n$ is identified with $H_0$. 

As in Theorem~\ref{thm:AlfsArePaths}, we can construct a loop in the contact cut graph $CC(\Sigma,d)$ representing a diagram for the achiral multisection with divides. 
In this case, note that since $H_0$ and $H_n$ are the same standard contact handlebody, there is a sequence of contact type $0$ moves closing the loop.


\end{proof}



%

\section{Stabilizations}
\label{s:Stab}

Given an achiral Lefschetz fibration $\mathcal{F}$ with bounded fibers $\pi: X^4 \to D^2$ we can obtain a new Lefschetz fibration of $X$ as follows: attach a $1$--handle so that its attaching spheres are attached to $\partial X$  along neighborhoods of points in the boundary of a regular fiber. Attaching the $1$--handle in such a way has the result of attaching a $2$--dimensional $1$--handle to each fiber of $\mathcal{F}$. Next attach a canceling $2$--handle so that the attaching sphere is attached along a curve in a regular fiber, which necessarily goes over the new $2$--dimensional $1$--handle exactly once, where the framing differs from the fiber framing by $\pm1$. The resulting $4$-manifold is diffeomorphic to $X$ and is supported by a new Lefschetz fibration $\mathcal{F}'$ whose fibers differ from those of $\mathcal{F}$ by an additional $2$--dimensional $1$--handle and an additional Lefschetz singularity. The newly added $2$--handle corresponds to the new Lefschetz critical point, the attaching sphere gives us the vanishing cycle, and the framing of the $2$--handle determines the chirality of the Lefschetz singularity.

\begin{definition}
	The process described above is called a Lefschetz stabilization of $\mathcal{F}$. If the $2$--handle is $-1$--framed, we say it is a \emph{positive Lefschetz stabilization}. If it is $+1$--framed, we call it a \emph{negative Lefschetz stabilization}.
\end{definition}

If we have a Lefschetz fibration over the disk, it supports a Weinstein domain (Theorem~\ref{thm:LFWeinsteincorrespondence}). If we perform a positive stabilization, the resulting Lefschetz fibration supports a Weinstein domain which is Weinstein homotopic to the first. The Weinstein homotopy is a cancellation of the new 1-handle and 2-handle.

Any achiral Lefschetz fibration $\mathcal{F}$ of $X$ induces an open book decomposition $\partial \mathcal{F}$ of $\partial X$ such that the monodromy is given by a composition of Dehn twists of the regular fiber of $\mathcal{F}$ along the vanishing cycles. We obtain a positive (resp. negative) Dehn twist if the vanishing cycle corresponds to a negatively (resp. positively) framed $2$--handle. Thus, a positive/negative Lefschetz stabilization gives rise to a positive/negative \emph{Hopf stabilization} of the induced open book. We can view Hopf stabilization abstractly as attaching a handle to the page of the open book and performing an additional Dehn twist along any simple closed curve which goes over the new handle exactly once.

\begin{remark}
	It is important to note that a Hopf stabilization of a given open book can be done in many ways. The first choice is how the topology of the page changes; the new handle can be attached along the same boundary component or different boundary components. Choosing the same boundary component increases the genus of the page by one, while decreasing the number of boundary components by one. On the other hand, choosing different boundary components preserves the genus of the page and increases the number of boundary components by one. The second choice is the isotopy class of the curve along which the Dehn twist is performed.
\end{remark}

A positive stabilization of a Lefschetz fibration can be doubled as in~\cite[Definition 5.2]{IslambouliStarkston}, to give a stabilization of a multisection with divides.

Given a path $P \subset CC(\Sigma, d)$ corresponding to a Lefschetz fibration, we can obtain a new path $P'\subset CC(\Sigma',d')$ by performing a Lefschetz stabilization:



\begin{proposition}\label{prop:stab}
	Let $\pi: X \to D^2$ be a Lefschetz fibration with fiber surface $\Sigma^+$ and let $\Sigma$ be the double of $\Sigma^+$. Suppose $\pi': X \to D^2$ is a Lefschetz fibration obtained by stabilizing $\pi$ along a properly embedded arc $a \subset \Sigma^+$, and let $(\Sigma', d')$ be the surface with divides obtained by doubling the fiber surface of $\pi'$. If $P = (v_1, v_2, ... v_n)$ is a path of length $n$ in $CC(\Sigma, d)$ corresponding to $\pi$, then $\pi'$ corresponds to a path $P'$ of length $n+m+1$ in $CC(\Sigma', d')$ where $m = |a \cap v_n|$.
\end{proposition}

\begin{proof}
	Let $\phi$ be the global monodromy of the Lefschetz fibration $\pi$ corresponding to the path $P= (v_1, v_2, ... v_n)$. The Lefschetz stabilization along $a \subset \Sigma^+$ gives rise to a Lefschetz fibration $\pi'$ with monodromy $\tau_c\circ \phi$, where $\psi$ is a Dehn twist along the curve $c$ which is a union of the arc $a$ with the core of the new $2$--dimensional $1$--handle which gives rise to the fiber surface of $\pi'$.
	
	First we note that by including $\Sigma^+$ into the new fiber surface, for each $i = 1, \ldots, n$ we obtain a vertex in $CC(\Sigma',d')$ which corresponds to $v_i$. This is equivalent to obtaining a cut system for the stabilized surface from a cut system of $\Sigma^+$ by using the co-core of the newly added $2$--dimensional $1$--handle. Note that $c$ intersects the additional cutting arc exactly once by construction. Additionally, an arc slide or Dehn twist performed in $\Sigma^+$ can also be realized in the new fiber surface. Thus, an edge of type $0$ or type $1$ between $v_i$ and $v_{i+1}$ gives rise to an edge of the same type in $CC(\Sigma',d')$ between corresponding vertices. 
	
	Denote the cut system of arcs which corresponds to $v_n'$ by $\mathcal{A}'_n$. Let $p_0$ be the intersection point between $c$ and the co-core of the stabilizing $1$--handle. An orientation of $c$ give rise to an ordering of the points of intersection of $c$ with $\mathcal{A}'_n$, denoted $(p_0, p_1, \ldots, p_m)$. Perform a type 0 move on the arc whose intersection with $c$ is $p_1$ by following $c$ in the reverse orientation and arc-sliding over the co-core of the new handle. This removes the intersection point $p_1$ and takes us to a new vertex $v'_{n+1}$. We can remove each subsequent intersection point by performing the same type $0$ move on the arc corresponding to each $p_i$ which gives rise to an additional vertex $v'_{n+i}$. Once $c\cap \mathcal{A}_n' = p_0$, we can realize $\tau_c^{\pm}$ by a type $1$ move. We have constructed a path $P' = (v'_1, \ldots, v'_n, v'_{n+1}, \ldots, v'_{n+m}, v'_{n+m+1})$ corresponding to $\pi'$, where the final vertex is the result of the type $1$ move (i.e. $\tau_c$) on $v'_{n+m}.$
\end{proof}

\begin{corollary}
	Let $(W,\omega,V)$ be a Weinstein domain. Then we have \[ \cL_g(W,\omega,V) \ge \cL_{g+1}(W,\omega,V).\]
\end{corollary}
\begin{proof}
	Let $P=(v_1,\dots,v_n)$ be a path in $CC(\Sigma,d)$ for some $\Sigma$ of genus $g$, realizing $\cL_g(W,\omega,V)$ and let $\pi:W\to D^2$ be a Lefschetz fibration corresponding to $P$. Since $v_n$ is a complete arc system when restricted to $\Sigma^+$, we can always choose a stabilization arc $a\subset \Sigma^+$ disjoint from $v_n$. 
	Stabilize $\pi$ along arc $a$ to get $\pi'$ and the corresponding path $P'\in CC(\Sigma',d')$, which has length $n$ by Proposition \ref{prop:stab}. As $g(\Sigma')=g+1$, we have
	$$\cL_{g+1}(W,\omega,V)\le \textrm{ length of $P'$} =\cL_g(W,\omega,V).$$ 
\end{proof}

\begin{corollary}
	Let $P\subset C(\Sigma, d)$ be a path which defines an open book decomposition $\ob$. Given any stabilization $\ob'$ of $\ob$, there is a path $P' \subset C(\Sigma', d')$ whose endpoints correspond  to $\ob'$.
\end{corollary}




\section{Weinstein domains with Weinstein $\mathcal{L}$-invariant equal to zero}
\label{s:L0}
Thinking of $\cL$ as a complexity measure, the ``simplest'' multisections with divides are those with $\cL=0$. In this section, we will study the multisections with divides with $\cL_{(\Sigma,d)}(\mathfrak{M})=0$, and the Weinstein domains they support. This amounts to interpreting paths in the contact cut graph $CC(\Sigma,d)$ whose edges are all type $1$.

First, we recall some basic properties of arc systems for a surface with boundary, $F$. (We will apply this in the case $F=\Sigma^+$.) By definition, an arc system $\mathcal{A}=\{a_1,\dots, a_n\}$ for $F$ cuts $F$ into a disk. We identify this disk with a polygonal presentation where each arc appears twice in the boundary of the polygon (and gluing these pairs together gives back $F$). For each arc $a_i\in \mathcal{A}$, let $B_i$ denote the simple closed curve in $F$ whose image in the polygon is a straight line segment connecting the midpoints of the two instances of $a_i$. Observe that $B_i$ intersects $a_i$ in $F$ at a single point, and $B_i$ is disjoint from $a_j\in \mathcal{A}$ for $j\neq i$. The collection of curves $\mathcal{B}=\{B_1,\dots, B_n\}$ is called the set of ``dual curves'' to $\mathcal{A}$. Note that this set is uniquely characterized up to isotopy by the property that the geometric intersection number of $B_i$ with $a_j$ is $1$ if $i=j$ and $0$ otherwise. From the polygonal presentation we can observe that for $i\neq j$, $B_i$ and $B_j$ either intersect once or are disjoint (since they are straight line segments in the polygon). Note, different choices of arc systems can yield different intersection patterns among the curves in $\mathcal{B}$.

The following lemma will help us translate a path in $CC(\Sigma,d)$ with no type $0$ moves into a Lefschetz fibration.

\begin{lemma}\label{l:L0LF}
	Let $F$ be a surface with genus $g$ and $b$ boundary components, and let $m=2g+b-1$. Let $P$ be a path in $CC(\Sigma,d)$ (where $\Sigma = F\cup_{\partial} \overline{F}$, and $d=\partial F$) built from exactly $n$ edges of contact type $1$ and none of contact type $0$. 
	
	Then there exists a collection of $m$ simple closed curves $\mathcal{B}=\{B_1,\dots, B_m\}$ in $F$, and a subset $J\subseteq \{1,\dots, n\}$ such that $W(P)$ is supported by a Lefschetz fibration with fiber $F$ and vanishing cycles $V_1,\dots, V_n$ where
		$$V_i \in  \left(\prod_{j\in J, j<i} \tau_{V_j}  \right)(\mathcal{B})$$
	where the order of the product is smallest $j$ to largest $j$ going right to left using composition of functions as the product operation.
	(In particular if $J=\{1,\dots, n\}$, $V_i \in (\tau_{V_{i-1}}\circ \cdots \tau_{V_1})(\mathcal{B})$.)
\end{lemma}

\begin{proof}
	Let $(v_0,v_1,\dots, v_n)$ denote the path of vertices in $CC(\Sigma,d)$ of contact type $1$ edges in the path $P$. Denote the curves in the cut system corresponding to $v_i$ by $C_1^i,\dots, C_m^i$.
	
	Via a diffeomorphism of $\overline{F}$ fixing the boundary, we can assume that $C_1^0,\dots, C_m^0$ is the double of an arc system $\mathcal{A}=(a_1,\dots, a_m)$ for $F$.
	
	Let $\mathcal{B}=\{B_1,\dots, B_m\}$ be the set of dual curves for $\mathcal{A}$. Let $\mathcal{B}'=\{B_1',\dots, B_m'\}$ denote the copies of these curves on the $\overline{F}$ side.
	
	Since $C_1^1,\dots, C_m^1$ is related to $C_1^0,\dots,C_m^0$ by a contact type 1 move, there exists a closed curve $T_1\subset \Sigma$ disjoint from $d$ such that 
	\begin{enumerate}
		\item $C_j^1 = \tau_{T_1}C_j^0$ for $j=1,\dots, m$, and
		\item $T_1\cap (C_1^0\cup \dots \cup C_m^0)$ is a single point.
	\end{enumerate}
	
	Since $T_1$ is disjoint from $d$, it is either contained in the interior of $F$ or the interior of $\overline{F}$. The requirement that $T_1$ intersects $C_1^0\cup\dots \cup C_m^0$ in a single point implies $T_1\in \mathcal{B}\cup\mathcal{B}'$.
	
	Similarly, since $v_i$ is related to $v_{i-1}$ by a contact type 1 move, there exists a closed curve $T_i\subset \Sigma$ disjoint from $d$ such that 
	\begin{enumerate}
		\item $C_j^i = \tau_{T_i}C_j^{i-1}$ for $j=1,\dots,m$, and
		\item $T_i\cap (C_1^{i-1}\cup\dots \cup C_m^{i-1})$ is a single point.
	\end{enumerate}
	
	We claim that the condition that $T_i$ intersects the cut system $C_1^{i-1} \cup \dots \cup C_m^{i-1}$ in a single point implies that $T_i\in \tau_{T_{i-1}}\circ \cdots \circ \tau_{T_1}(\mathcal{B}\cup\mathcal{B}')$. This is because $\tau_{T_{i-1}}\circ \cdots \circ \tau_{T_1}$ is a diffeomorphism which takes $C_1^0,\dots, C_m^0$ to $C_1^{i-1},\dots, C_m^{i-1}$, so it takes the dual curves to $C_1^0,\dots, C_m^0$ to the dual curves of $C_1^{i-1},\dots, C_m^{i-1}$.
	
	Next, we will recover a Lefschetz fibration for the Weinstein domain $W(P)$.
	By Lemma~\ref{l:LFMDdiagrams}, the vanishing cycles are determined by the $T_i$ by 
	\begin{itemize}
		\item $V_i = T_i$ if $T_i\subset F$ or 
		\item $V_i = \tau_{V_{i-1}}\circ \cdots \tau_{V_1}(T_i)\subset F$ if $T_i\subset \overline{F}$.
	\end{itemize}
	
	Let $J\subseteq \{1,\dots, n\}$ be the subset of indices such that $T_i\subset F$. Since $T_1\in \mathcal{B}\cup\mathcal{B}'$, $V_1\in \mathcal{B}$ (note that $B_1',\dots, B_m'$ are just copies of $B_1,\dots, B_m$ in the $\overline{F}$ side, and that for $i=1$ the formula for $V_1$ is the same whether $T_1\subset F$ or $T_1\subset \overline{F}$). This establishes a base case. Now we will prove that for $i>1$,
	$$V_i \in  \left(\prod_{j\in J, j<i} \tau_{V_j}  \right)(\mathcal{B}).$$
	
	Let $\mathcal{B}_j = \tau_{T_j}\circ \cdots \circ \tau_{T_1}(\mathcal{B})$ and $\mathcal{B}_j' = \tau_{T_j}\circ \cdots \circ \tau_{T_1}(\mathcal{B}')$. 
	
	Case 1: If $i\in J$ then 
	$$V_i=T_i\in \mathcal{B}_{i-1}=\tau_{T_{i-1}}\circ \cdots \circ \tau_{T_1}(\mathcal{B}) = \left(\prod_{j\in J, j<i} \tau_{V_j}  \right)(\mathcal{B})$$
	since for $j\notin J$, the Dehn twist along $T_j$ acts as the identity on the $F$ side.
	
	Case 2: For $i\notin J$, $T_i\in \mathcal{B}_{i-1}'$. 
	
	We will prove inductively in $i$ that 
	$$\tau_{V_{i-1}}\circ \cdots \circ \tau_{V_1}(\mathcal{B}_{i-1}') = \left(\prod_{j\in J, j<i} \tau_{V_j}  \right)(\mathcal{B})$$
	(after identifying $\overline{F}$ with $F$).
	Since $T_i\subset \overline{F}$, by Lemma~\ref{l:LFMDdiagrams}, $V_i = (\tau_{V_{i-1}}\circ \cdots \circ \tau_{V_1})(T_i)\in \tau_{V_{i-1}}\circ \cdots \circ \tau_{V_1}(\mathcal{B}_{i-1}')$, so the result will follow once we prove this claim. 
	
	If $i-1\in J$, then $\mathcal{B}_{i-1}' = \tau_{T_{i-1}}(\mathcal{B}_{i-2}') = \mathcal{B}_{i-2}'$ since the Dehn twist along $T_{i-1}$ is performed on the $F$ side. 
	
	Therefore
	$$\tau_{V_{i-1}}\circ \cdots \circ \tau_{V_1}(\mathcal{B}_{i-1}') = \tau_{V_{i-1}}\circ \cdots \circ \tau_{V_1}(\mathcal{B}_{i-2}').$$
	By inductive hypothesis
	$$\tau_{V_{i-2}}\circ \cdots \circ \tau_{V_1}(\mathcal{B}_{i-2}') = \left(\prod_{j\in J, j<i-1} \tau_{V_j}  \right)(\mathcal{B}).$$
	Therefore
	$$\tau_{V_{i-1}}\circ \cdots \circ \tau_{V_1}(\mathcal{B}_{i-1}') = \tau_{V_{i-1}}\circ \left(\prod_{j\in J, j<i-1} \tau_{V_j}  \right)(\mathcal{B}) = \left(\prod_{j\in J, j<i} \tau_{V_j}  \right)(\mathcal{B}).$$
	
	If $i-1\notin J$, then with respect to the $F$ orientation $\mathcal{B}_{i-1}' = \tau_{T_{i-1}}^{-1}(\mathcal{B}_{i-2}')$.
	
	Using the conjugation formula
	$$\tau_{T_{i-1}}^{-1} = \tau_{(\tau_{V_{i-2}}\circ \cdots \circ \tau_{V_1})^{-1}(V_{i-1})}^{-1} = (\tau_{V_{i-2}}\circ \cdots \circ \tau_{V_1})^{-1}\tau_{V_{i-1}}^{-1}(\tau_{V_{i-2}}\circ \cdots \circ \tau_{V_1}).$$
	
	Therefore
	$$(\tau_{V_{i-1}}\circ \cdots \circ \tau_{V_1})(\mathcal{B}_{i-1}') = (\tau_{V_{i-1}}\circ \cdots \circ \tau_{V_1}) \tau_{T_{i-1}}^{-1}(\mathcal{B}_{i-2}')$$
	$$ = (\tau_{V_{i-1}}\circ \cdots \circ \tau_{V_1}) (\tau_{V_{i-2}}\circ \cdots \circ \tau_{V_1})^{-1}\tau_{V_{i-1}}^{-1}(\tau_{V_{i-2}}\circ \cdots \circ \tau_{V_1}) (\mathcal{B}_{i-2}')$$
	$$=(\tau_{V_{i-2}}\circ \cdots \circ \tau_{V_1}) (\mathcal{B}_{i-2}')$$
	which is equal to
	$$\left(\prod_{j\in J, j<i-1} \tau_{V_j}  \right)(\mathcal{B})$$
	by inductive hypothesis. Since $i-1\notin J$, this is the same as
	$$\left(\prod_{j\in J, j<i} \tau_{V_j}  \right)(\mathcal{B}).$$
	
\end{proof}

\begin{definition}\label{def:visibleinvisible}
	There are two different kinds of vanishing cycles in the Lefschetz fibrations constructed in the previous lemma. The $V_i$ so that $i\in J$ coming from $T_i$ on the $\Sigma^+=F$ side which we will refer to as \emph{visible} curves, and the $V_j$, where $j\notin J$ coming from $T_i$ on the $\Sigma^-=\overline{F}$ side, which we will refer to as \emph{invisible} curves. 
\end{definition}

Now we begin a classification of Weinstein domains supported by multisections with divides with $\cL=0$.
First we consider the simplest case of $(\Sigma,d)$, where $\Sigma$ is a genus $1$ surface divided by $d$ into two annuli.

\begin{proposition}\label{p:L0annulus}
	Suppose $F$ is an annulus. Let $(\Sigma,d) = (F\cup_\partial \overline{F},\partial F)$. Let $(W,\omega,V)$ be a Weinstein domain supported by a multisection with divides $\mathfrak{M}$ with core surface $(\Sigma,d)$ and $n$ sectors. If $\mathcal{L}_{(\Sigma,d)}(\mathfrak{M}) = 0$, then $X$ is a linear plumbing of $n-1$ copies of $D^*S^2$.
\end{proposition}

\begin{proof}
	Since $\mathcal{L}_{(\Sigma,d)}(\mathfrak{M})=0$, there is an oriented path $P$ in $CC(\Sigma,d)$ such that $\mathfrak{M}(P)=\mathfrak{M}$ whose edges are all contact type 1 moves. Then $(W,\omega,V)$ is supported by a Lefschetz fibration with annulus fiber and vanishing cycles as in Lemma~\ref{l:L0LF}. In this case the set $\mathcal{B}$ is a single curve $\alpha$, the core of the annulus. Since $\tau_\alpha(\mathcal{B})=\mathcal{B}$, each type 1 move corresponds to a right-handed Dehn twist about the core $\alpha$ of the annulus, contributing a single vanishing cycle $\alpha$ to the Lefschetz fibration. 
	
	\begin{figure}
		\centering
		\includegraphics[scale=.2]{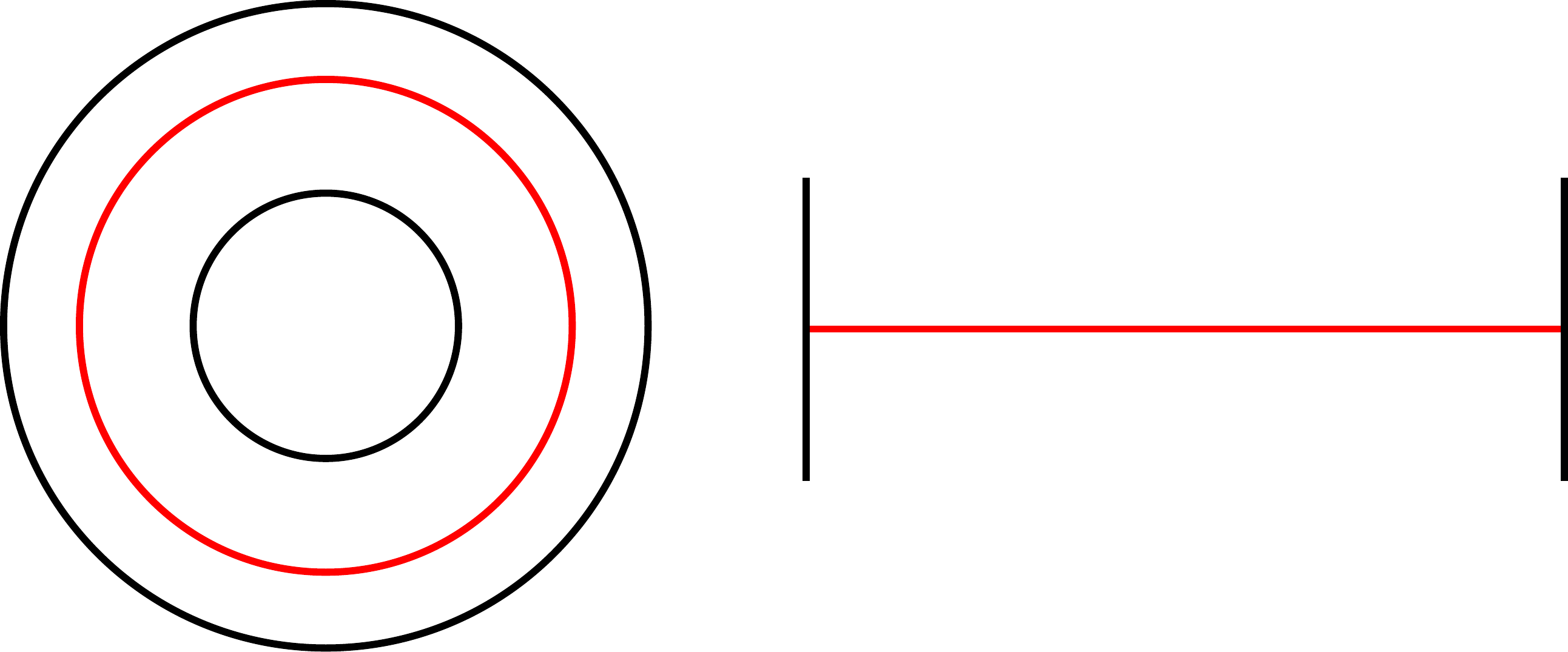}
		\caption{The curve $\alpha$ and its Legendrian lift $\Lambda_\alpha$.}
		\label{fig:LF_annulus_one_VC}
	\end{figure}
	
	Given a Lefschetz fibration with fiber $F$, a strategy was developed in~\cite{CasalsMurphy} to draw a front projection of the Legendrian attaching spheres corresponding to the vanishing cycles in $\#_k S^1\times S^2$, where $k=b_1(F)$. The key is to find the Legendrian lift of each vanishing cycle.
	
	In this case, $b_1(F)=1$, so there is a unique $1$-handle, denoted by two vertical bars. The Legendrian lift of the vanishing cycle $\alpha$ is the curve $\Lambda_\alpha$ depicted in Figure~\ref{fig:LF_annulus_one_VC}. 
	
	If the oriented path corresponding to $X$ consists of $n$ edges, then the Weinstein presentation of $X$ consists of $n$ parallel copies of $\Lambda_\alpha$, each offset slightly in the Reeb direction. This is depicted in Figure~\ref{fig:WD_repeated_alpha}. Also depicted in this figure is a sequence of handleslides, then isotopies, then a single handle cancellation transforming the presentation into a link of unknots. The sequence of handleslides and isotopies depicted in Figure~\ref{fig:WD_chain_of_unknots} demonstrates how to simplify this link into a linear chain of unknots. 
	
	Because a Weinstein presentation for $D^*S^2$ is represented by a single Legendrian unknot with $tb=-1$, and linear chain of such Legendrian unknots is a Weinstein presentation for a plumbing of copies of $D^*S^2$, yielding the result.
	
	\begin{figure}
		\centering
		\includegraphics[scale=.2]{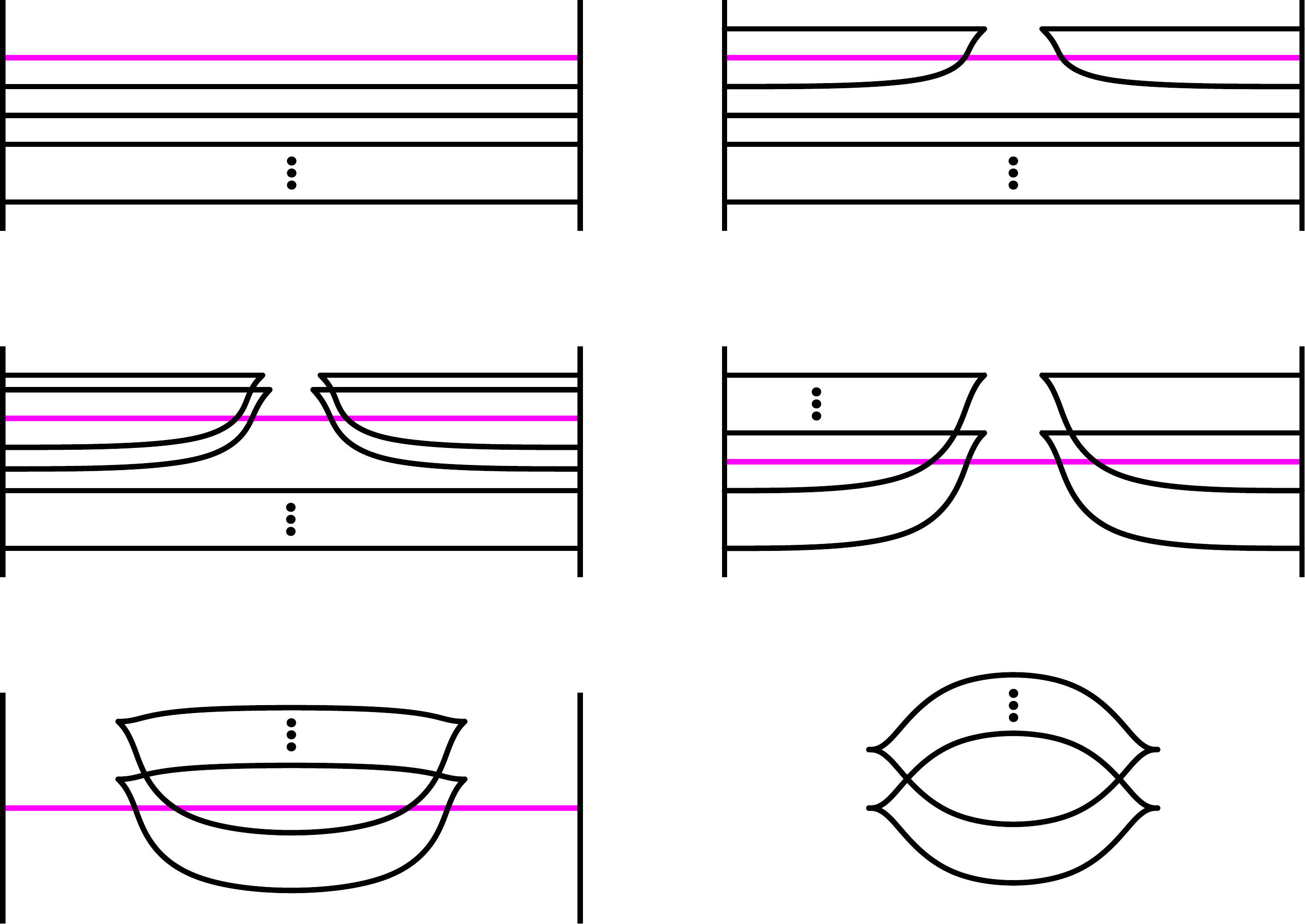}
		\caption{First we handleslide each black 2-handle over the pink 2-handle. Then we move the black unknots through the 1-handle. Finally we cancel the pink 2-handle with the 1-handle to obtain a link of $n-1$ Legendrian unknots.}
		\label{fig:WD_repeated_alpha}
	\end{figure}
	
	\begin{figure}
		\centering
		\includegraphics[scale=.2]{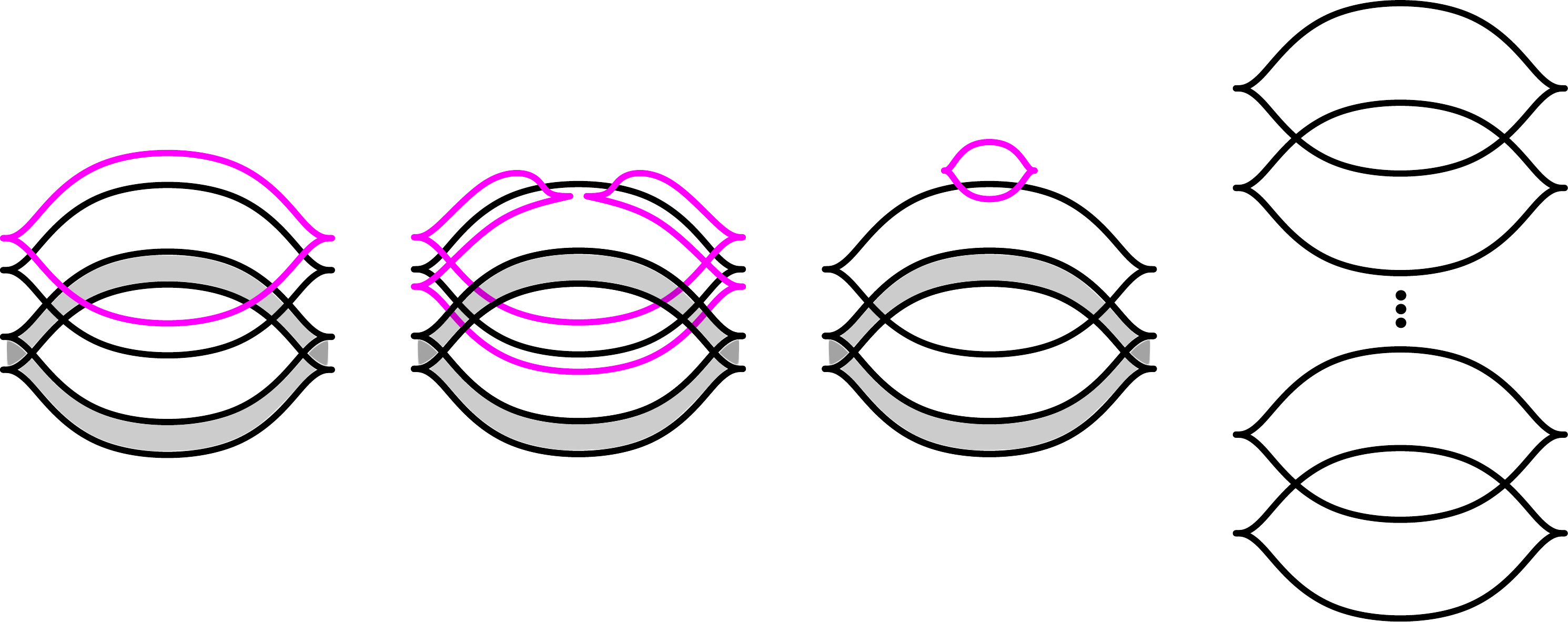}
		\caption{The grey band represents some number of parallel unknots. A handleslide and a series of isotopies unlinks the topmost pink unknot from all but all but one other unknot. The result of repeating this process on all of the parallel unknots is a chain of unknots.}
		\label{fig:WD_chain_of_unknots}
	\end{figure}

\end{proof}

\begin{theorem}\label{thm:L0Weinsteinhandles}
	Let $W$ be a Weinstein domain with $\cL(W)=0$. Then $W$ admits a Weinstein presentation where none of the $2$-handles pass through $1$-handles, and each $2$-handle is attached along a Legendrian unknot with $tb=-1$. (There can be nontrivial linking between the attaching spheres of the $2$-handles.)
\end{theorem}

%
%

\begin{proof}
	If $\cL(W)=0$, there exists a multisection with divides diagram for $W$ with $(\Sigma,d) = (F\cup_\partial \overline{F},\partial F)$ corresponding to an oriented path in $CC(\Sigma,d)$ with $n$ edges all of which are contact type 1. By a diffeomorphism of $\Sigma$ supported on $\overline{F}$, we assume that the first contact cut system in the path is the double of an arc system $\mathcal{A}_0$ on $F$.
	
	Let the genus of $F$ be $g$ and the number of boundary components $b$, and let $m=2g+b-1$. Let $\mathcal{B}=\{B_1,\dots, B_m\}$ denote the set of dual curves in $F$ to the arc system $\mathcal{A}$.
	
	
	By Lemma~\ref{l:L0LF}, $W$ is supported by a Lefschetz fibration with fiber $F$ and vanishing cycles $(V_1, \dots, V_n)$ where
	$$V_i \in  \left(\prod_{j\in J, j<i} \tau_{V_j}  \right)(\mathcal{B}).$$
	
	We will modify this sequence by certain Hurwitz moves. A Hurwitz move changes the ordered list $(V_1, \dots, V_i, V_{i+1},\dots, V_n)$ to the list $(V_1, \dots, \tau_{V_i}^{-1}(V_{i+1}), V_{i},\dots, V_n)$. Lefschetz fibrations related by Hurwitz moves are equivalent and thus support the same underlying Weinstein domain. We will choose Hurwitz moves so that at the end, all the vanishing cycles are elements of $\mathcal{B}$. 
	
	First we will replace the notation $V_1,\dots, V_n$ with an interspersed combination of  $U_1,\dots, U_l$ and $W_1,\dots, W_j$ where the $U_i$ are the invisible curves, and the $W_i$ are visible curves. We consider the entire sequence of vanishing cycles from left to right. If $U_1$ is the first invisible curve, we note that $U_1 = \tau_{W_{i-1}}(\tau_{W_{i-2}}(\dots(\tau_{W_1}(B_{s_1}))))$ where  $B_{s_1}\in \mathcal{B}$. We modify the sequence as follows:
	\begin{align*}
		&(W_1, W_2, \dots, W_{i-2}, W_{i-1}, U_1, V_{i+1} \dots, V_n)\\
		=& (W_1, W_2, \dots, W_{i-2}, W_{i-1}, \tau_{W_{i-1}}(\tau_{W_{i-2}}(\dots(\tau_{W_1}(B_{s_1})))),V_{i+1}, \dots, V_n)\\
		\mapsto &(W_1, W_2, \dots, W_{i-2}, \tau_{W_{i-2}}(\dots(\tau_{W_1}(B_{s_1}))), W_{i-1},V_{i+1}, \dots, V_n)\\
		&\dots\\
		\mapsto &(B_{s_1}, W_1, W_2, \dots, W_{i-2}, W_{i-1}, V_{i+1}, \dots, V_n)
	\end{align*}
	We repeat this procedure for each of the invisible curves $U_i$ from the smallest index to the largest, resulting in a new sequence of vanishing cycles of the form 
	$(B_{s_{1}}, B_{s_{2}}, \dots, B_{s_{l}},W_1,\dots,W_j)$.
	
	Next we note that $W_j = \tau_{W_{j-1}}(\tau_{W_{j-2}}(\dots(\tau_{W_1}(B_{r_j}))))$ for some $B_{r_j}\in \mathcal{B}$ and we modify the visible curves in the sequence from right to left as follows:
	\begin{align*}
		&(B_{s_{1}}, \dots, B_{s_{l}},W_1,\dots,W_{j-2}, W_{j-1}, W_j)  \\
		=&(B_{s_{1}},  \dots, B_{s_{l}},W_1,\dots,W_{j-2}, W_{j-1}, \tau_{W_{j-1}}(\tau_{W_{j-2}}(\dots(\tau_{W_1}(B_{r_j})))))\\
		\mapsto &(B_{s_{1}},  \dots, B_{s_{l}},W_1,\dots,W_{j-2}, \tau_{W_{j-2}}(\dots(\tau_{W_1}(B_{r_j}))), W_{j-1})\\
		&\dots\\
		\mapsto &(B_{s_{1}},\dots, B_{s_{l}},B_{r_j},W_1,\dots,W_{j-2}, W_{j-1})
	\end{align*}
	We repeat this procedure for each the visible curves $W_j$ from the largest index to the smallest, resulting in a final sequence of vanishing cycles of the form $(B_{s_{1}}, \dots, B_{s_{l}},B_{r_j},B_{r_{j-1}},\dots, B_{r_1})$ where every curve is in $\mathcal{B}$. 
	
	Now to obtain a Weinstein presentation corresponding to this Lefshetz fibration, we take the Legendrian lifts of all of these curves. Since all the curves are in $\mathcal{B}$, the front projections of the Legendrian lifts will each be a simple arc through a single $1$-handle with no cusps and no self-crossings. For each 1-handle which intersects a non-zero number of 2-handles, we may handleslide all but one of these 2-handles over the lowest 2-handle, then isotope these 2-handles through that 1-handle to detach them from the 1-handle, and then cancel the 1-handle with the remaining 2-handle. The result is a presentation consisting of $k$ 1-handles and $n-m+k$ 2-handles attached along a link whose components are all $tb=-1$ unknots, where $k$ is the number of curves in $\mathcal{B}$ not represented by $(B_{s_{1}}, \dots, B_{s_{l}},B_{r_j},\dots, B_{r_1})$. See Figures~\ref{fig:genus1_curves} and~\ref{fig:L0_genus1} for an example on a genus 1 surface. See Figures~\ref{fig:genus2_curves} and~\ref{fig:L0_genus2} for an example on a genus 2 surface. 
\end{proof}

\begin{figure}
	\centering
	\includegraphics[scale=0.2]{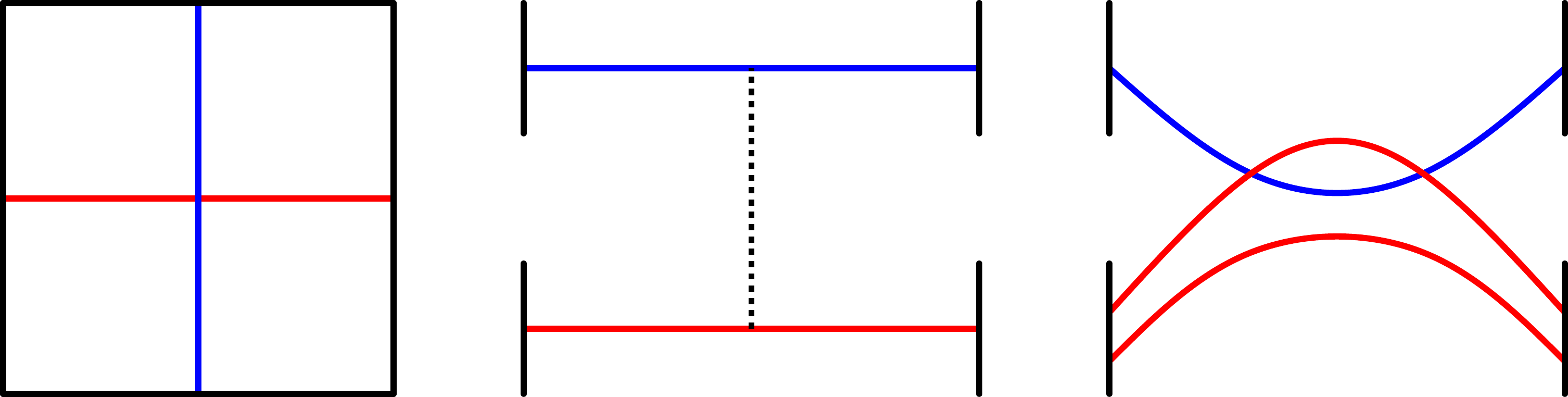}
	\caption{The curves $\alpha$ (red) and $\beta$ (blue) on a genus 1 surface, and their corresponding Legendrian lifts. The lifted curves may be linked near the dotted black line depending on the order they appear in the monodromy of the Lefshetz fibration. For instance in the rightmost figure, the monodromy consists of a Dehn twist about $\alpha$, then $\beta$, and then $\alpha$.}
	\label{fig:genus1_curves}
\end{figure}

\begin{figure}
	\centering
	\includegraphics[scale=0.2]{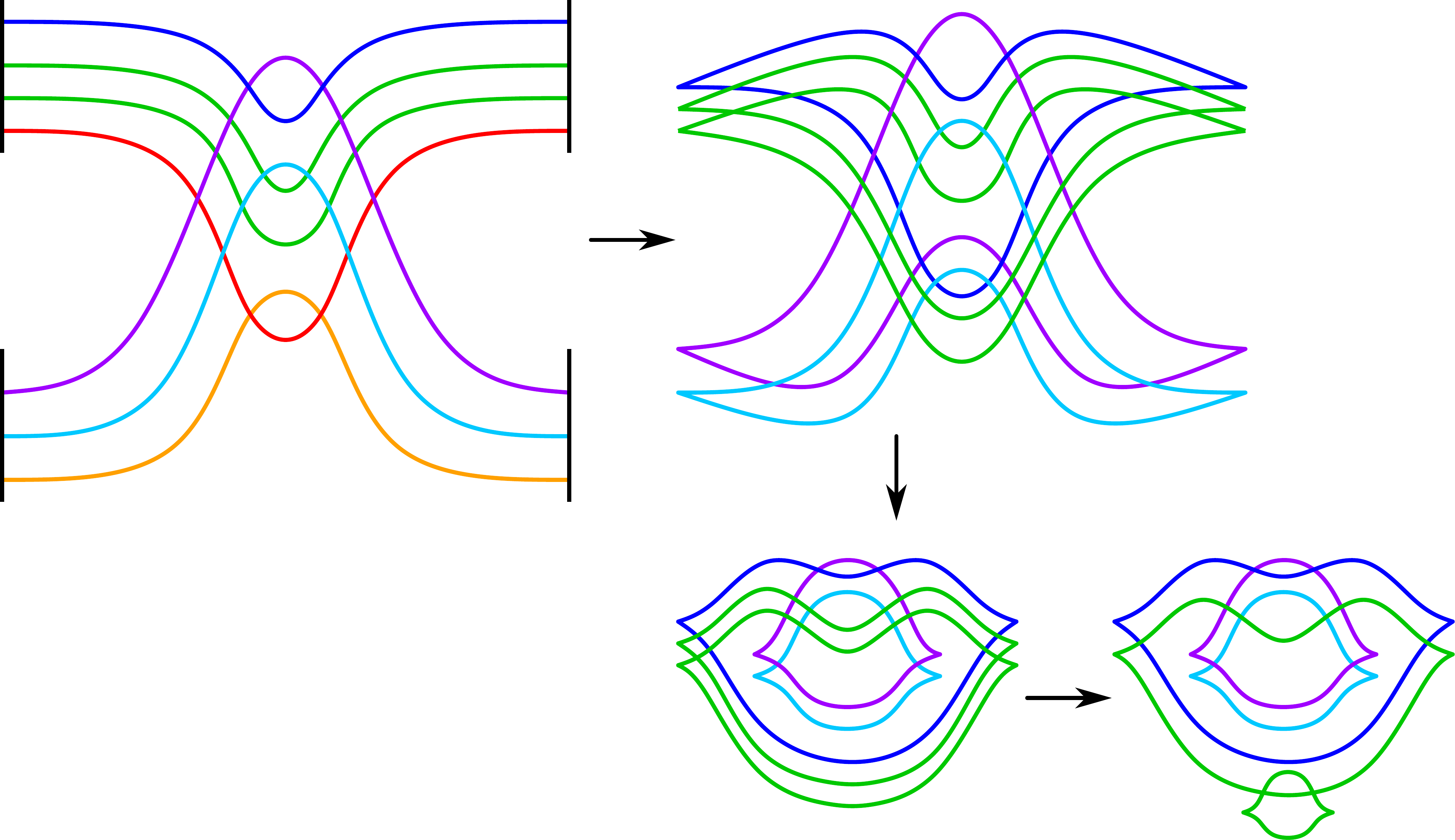}
	\caption{Construction of a Weinstein presentation for a Weinstein domain with $\mathcal{L}(W) = 0$ and $\Sigma$ of genus 1. The total monodromy of the Lefshetz fibration (after an application of relevant Hurwitz moves) is given by the vanishing cycles $(\beta,\alpha,\beta,\beta,\alpha,\beta,\alpha)$. The result is a link of five $tb=-1$ unknots.}
	\label{fig:L0_genus1}
\end{figure}

\begin{figure}
	\centering
	\includegraphics[scale=0.2]{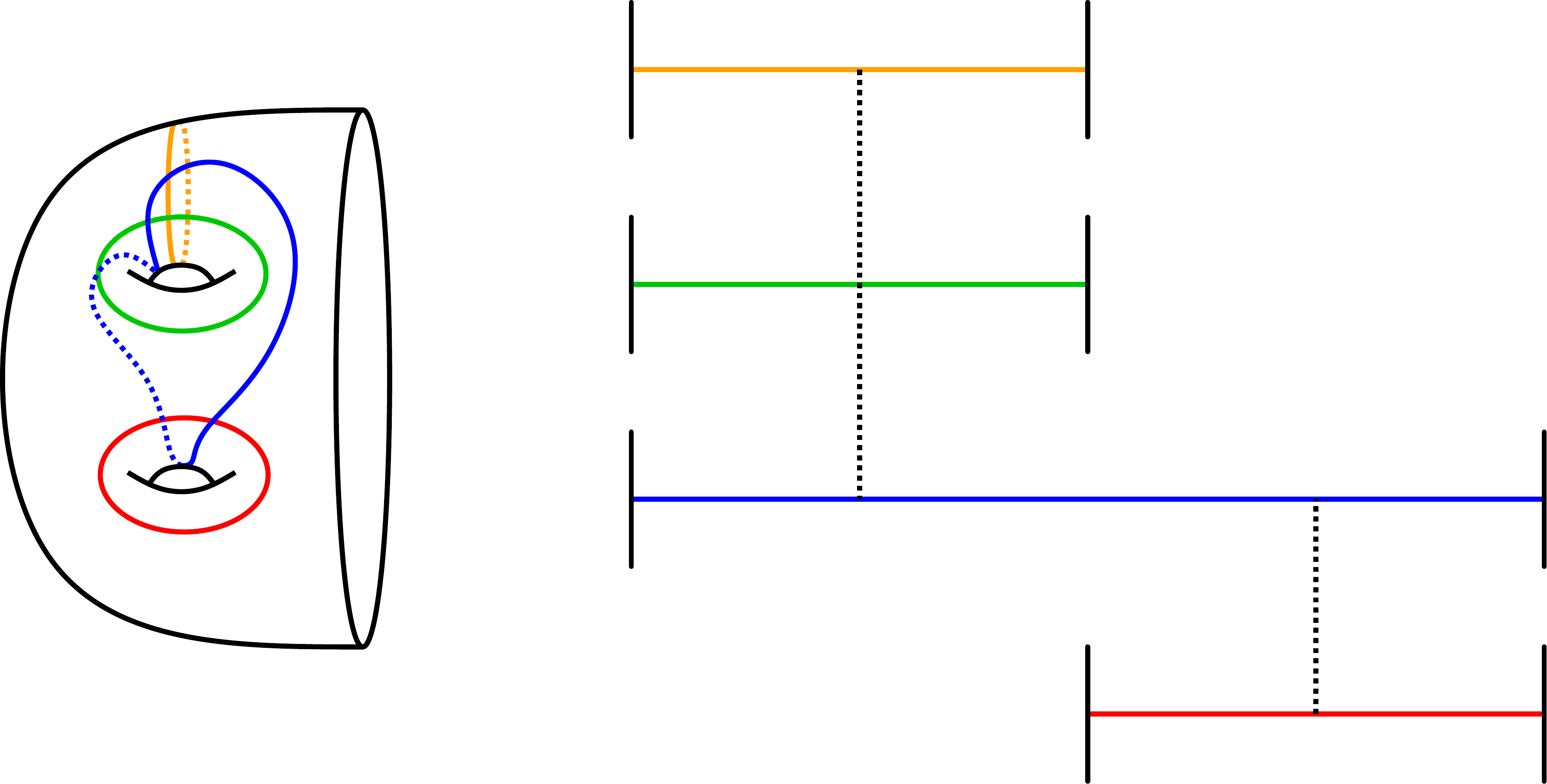}
	\caption{A choice of curves $\alpha$ (red), $\beta$ (blue), $\gamma$ (green), and $\delta$ (orange) on a genus 2 surface, and their corresponding Legendrian lifts. The lifted curves may be linked near the dotted black lines (determined by whether these curves intersect on the surface) depending on the order they appear in the monodromy of the Lefshetz fibration.}
	\label{fig:genus2_curves}
\end{figure}

\begin{figure}
	\centering
	\includegraphics[scale=0.2]{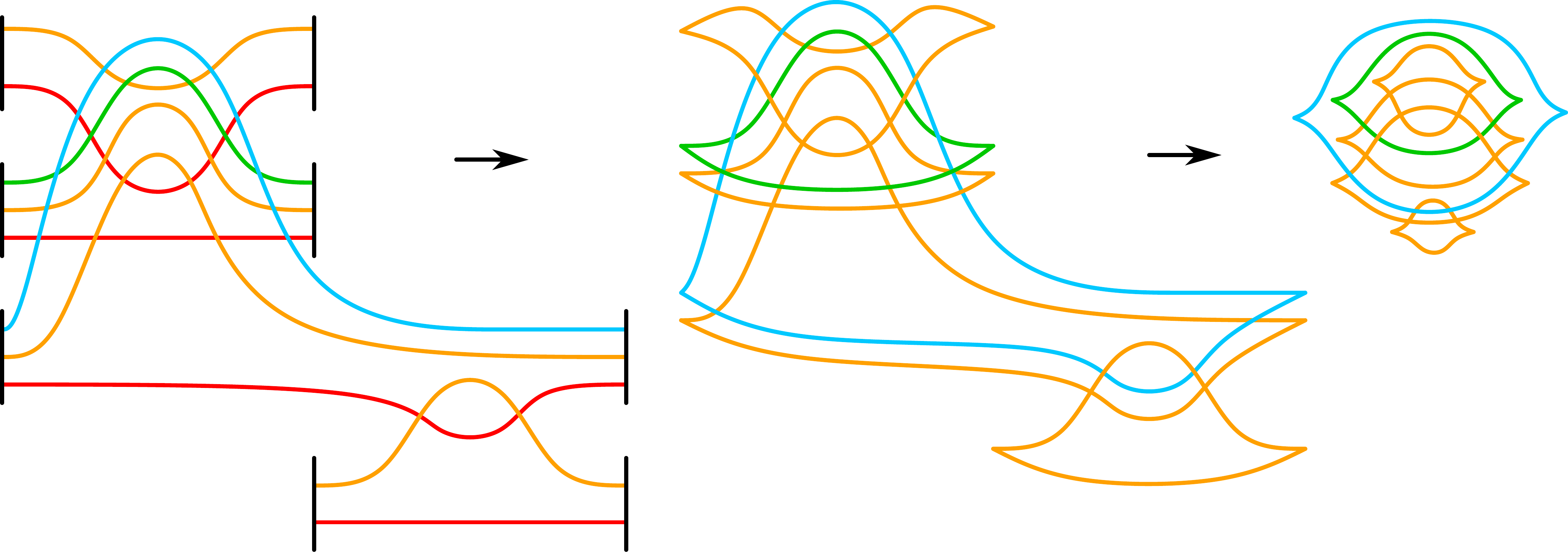}
	\caption{Construction of a Weinstein presentation for a Weinstein domain with $\mathcal{L}(W) = 0$ and $\Sigma$ of genus 2. The total monodromy of the Lefshetz fibration (after an application of relevant Hurwitz moves) is given by the vanishing cycles $(\alpha,\beta,\gamma,\delta,\alpha,\beta,\gamma,\delta,\gamma,\beta)$. The result is a link of six $tb=-1$ unknots.}
	\label{fig:L0_genus2}
\end{figure}

\begin{corollary}\label{cor:pi1}
	If $(W,\omega,V)$ is a Weinstein domain with $\cL = 0$, then $\pi_1(W)$ is a free group.
\end{corollary}

\begin{proof}
	This follows immediately from the handle presentation since none of the $2$-handle attaching spheres pass through the $1$-handles.
\end{proof}

A slightly stronger corollary that also follows immediately from Theorem~\ref{thm:L0Weinsteinhandles} is that $(W,\omega)$ splits as a boundary sum of some number of copies of $S^1\times D^3$ with a simply connected Weinstein domain.

\begin{corollary}\label{cor:L0c1}
	If $(W,\omega,V)$ is a Weinstein domain with $\cL = 0$, then $c_1(\omega)=0$.
\end{corollary}

\begin{proof}
	Given a Weinstein handle decomposition for a $4$-dimensional Weinstein domain, the first Chern class can be computed by~\cite[Proposition 2.3]{GompfHandlebody}. Specifically, when none of the $2$-handles pass through the $1$-handles, each of the $2$-handles corresponds to a $2$-cycle. The class $c_1(\omega)$ is represented by a cocycle which evaluates on each $2$-handle $h_i$ as the rotation number of the Legendrian attaching knot. Since Theorem~\ref{thm:L0Weinsteinhandles} tells us the Weinstein domain has a handle decomposition where all $2$-handles are attached along standard Legendrian unknots with $tb=-1$, their rotation number must be $0$. Thus $c_1(\omega)$ is represented by the zero cocycle.
\end{proof}

\begin{corollary} \label{cor:even}
	If $(W,\omega,V)$ is a Weinstein domain with $\cL = 0$, then the intersection form on $H_2(W)$ is even.
\end{corollary}

\begin{proof}
	This follows directly from Theorem~\ref{thm:L0Weinsteinhandles} since $H_2(W)$ is generated by 2-handles attached with framing $-2$ (even). It also follows from Corollary~\ref{cor:L0c1} since the first Chern class is characteristic.
\end{proof}

\subsection{Boundary sum decompositions and $\cL=0$}

In the smooth setting, closed $4$-manifolds with Kirby-Thompson $\cL$-invariant equal to zero satisfy a decomposition theorem: they decompose into connected summands of standard $4$-manifolds ($\cptwo$, $\cptwobar$, $S^1\times S^3$, $S^2\times S^2$). For Weinstein domains, we do not generally have this decomposition into basic summands. The reason for the difference is that up to diffeomorphism, one can always assume the first cut system in a trisection is standard. This is not the case with arc systems. Different arc systems cannot always be related by a diffeomorphism (this is easiest to see by noticing different intersection patterns for the dual curve systems to different arc systems). However, if the initial arc system respects the boundary sum decomposition, we do obtain a decomposition theorem, as we will now prove.

\begin{lemma}\label{l:L0decompose}
	Let $\mathfrak{M}$ be a multisection with divides diagram corresponding to a path in the contact cut graph $CC(\Sigma,d)$ with $\cL(\mathfrak{M})=0$. 
	
	Identify $\Sigma$ as $F\cup \overline{F}$ via a diffeomorphism such that the first contact cut system is the double of an arc system $\mathcal{A}_1$.
	Let $\mathcal{B}$ be the dual curve system to $\mathcal{A}_1$. 
	
	Suppose $\Gamma=\{\gamma_1,\dots, \gamma_d\}$ is a collection of arcs on $F$ which are disjoint from $\mathcal{A}_1\cup \mathcal{B}$. 
	
	Then all curves in $\mathfrak{M}$ are disjoint from $\Gamma$.
\end{lemma}

\begin{proof}
Since $\cL(\mathfrak{M})=0$, consecutive contact cut systems $C_i$ and $C_{i+1}$ in the path in $CC(\Sigma,d)$ are related by a right-handed Dehn twist about a curve $T_i$, which is disjoint from the dividing set $d$ and intersects $C_i$ in exactly one point. We will show that each $T_i$ is disjoint from $\Gamma$.

Since $C_1$ is the double of $\mathcal{A}_1$, the only curves which intersect $C_1$ in a single point and are disjoint from $d$ are the curves in $\mathcal{B}$ on the $F$ or $\overline{F}$ side. By assumption these are disjoint from $\Gamma$, so $T_1$ is disjoint from $\Gamma$.

For $i>1$, inductively assume that $T_1,\dots, T_{i-1}$ are disjoint from $\Gamma$. We have $C_i = \tau_{T_{i-1}}\circ \cdots \circ \tau_{T_1}(C_1)$, so the only curves which intersect $C_i$ in a single point and are disjoint from $d$ are the curves in $\tau_{T_{i-1}}\circ \cdots \circ \tau_{T_1}(\mathcal{B}\cup \overline{\mathcal{B}})$. Since $T_1,\dots, T_{i-1}$ and $\mathcal{B}\cup \overline{\mathcal{B}}$ are assumed to be disjoint from $\Gamma$, and Dehn twists are supported in a small neighborhood of the curve we are twisting along, $\tau_{T_{i-1}}\circ \cdots \circ \tau_{T_1}(\mathcal{B}\cup \overline{\mathcal{B}})$ is disjoint from $\Gamma$. Thus $T_i$ is disjoint from $\Gamma$ for all $i$, so $C_{i+1} = \tau_{T_{i}}\circ \cdots \circ \tau_{T_1}(C_1)$ is disjoint from $\Gamma$ for all $i$.

\end{proof}

\begin{figure}
	\centering
	\includegraphics[scale=.2]{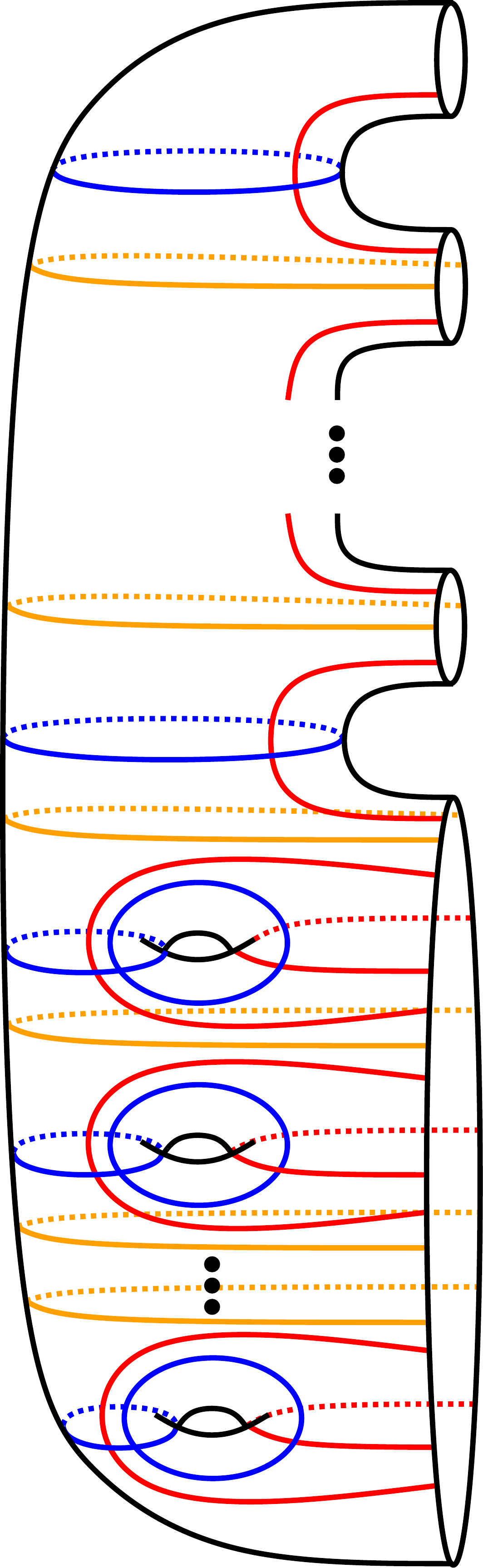}
	\caption{Surface $F$ with red arcs $\mathcal{A}_1=\{a_1^1,\dots, a_1^k\}$, dual blue curves $V_1,\dots, V_k$, and orange separating arcs $\gamma_1,\dots, \gamma_{g+b-2}$.}
	\label{fig:bdrysum}
\end{figure}

\begin{corollary}\label{cor:L0decomp}
	Suppose $\mathfrak{M}$ is a multisection with divides diagram corresponding to a path in the contact cut graph $CC(\Sigma,d)$ with $\cL(\mathfrak{M})=0$, such that the first contact cut system is the double of the standard arc system shown in Figure~\ref{fig:bdrysum}.
	
	Then the Weinstein domain $X(\mathfrak{M})$ is a boundary sum of Weinstein domains supported by multisections with divides where $F_i$ an annulus or genus $1$ surface with a single boundary component.
\end{corollary}

\section{Multisections with divides with arbitrarily large Weinstein $\mathcal{L}$-invariant}
\label{ss:largeL}

Now we will consider the other extreme--examples where the Weinstein $\cL$-invariant grows arbitrarily large. We will give two families of examples of multisections with divides (coming from Lefschetz fibrations) where $\cL_{(\Sigma_n,d)}(\mathfrak{M}_n)$ grows arbitrarily large. We suspect that the Weinstein $\cL$-invariants of the underlying Weinstein domains also grow arbitrarily large and make some remarks about strategies towards proving this at the end of this section.

\begin{examples} \label{ex:large1}
	Let $F_n$ be a planar surface with $2n+1$ boundary components. We will label the boundary components $\{0,1,\dots, 2n\}$, and draw $F_n$ so that $0$ is the ``outer'' boundary component, and the other boundary components are the boundaries of holes, ordered left to right as in Figure~\ref{fig:largeLF}. Consider the Lefschetz fibration $LF_n$ with fiber $F_n$ and vanishing cycles $(a_1,b_1,c_1,\dots, a_n,b_n,c_n)$, where $a_i$ is boundary parallel to boundary component $2i-1$, $b_i$ is boundary parallel to boundary component $2i$, and $c_i$ is a convex simple closed curve enclosing the $2i-1$ and $2i$ holes as in Figure~\ref{fig:largeLF}. All of the vanishing cycles are disjoint, so their order does not matter (i.e. Hurwitz moves do not change the set of vanishing cycle curves).
	
	\begin{figure}
		\centering
		\includegraphics[scale=.2]{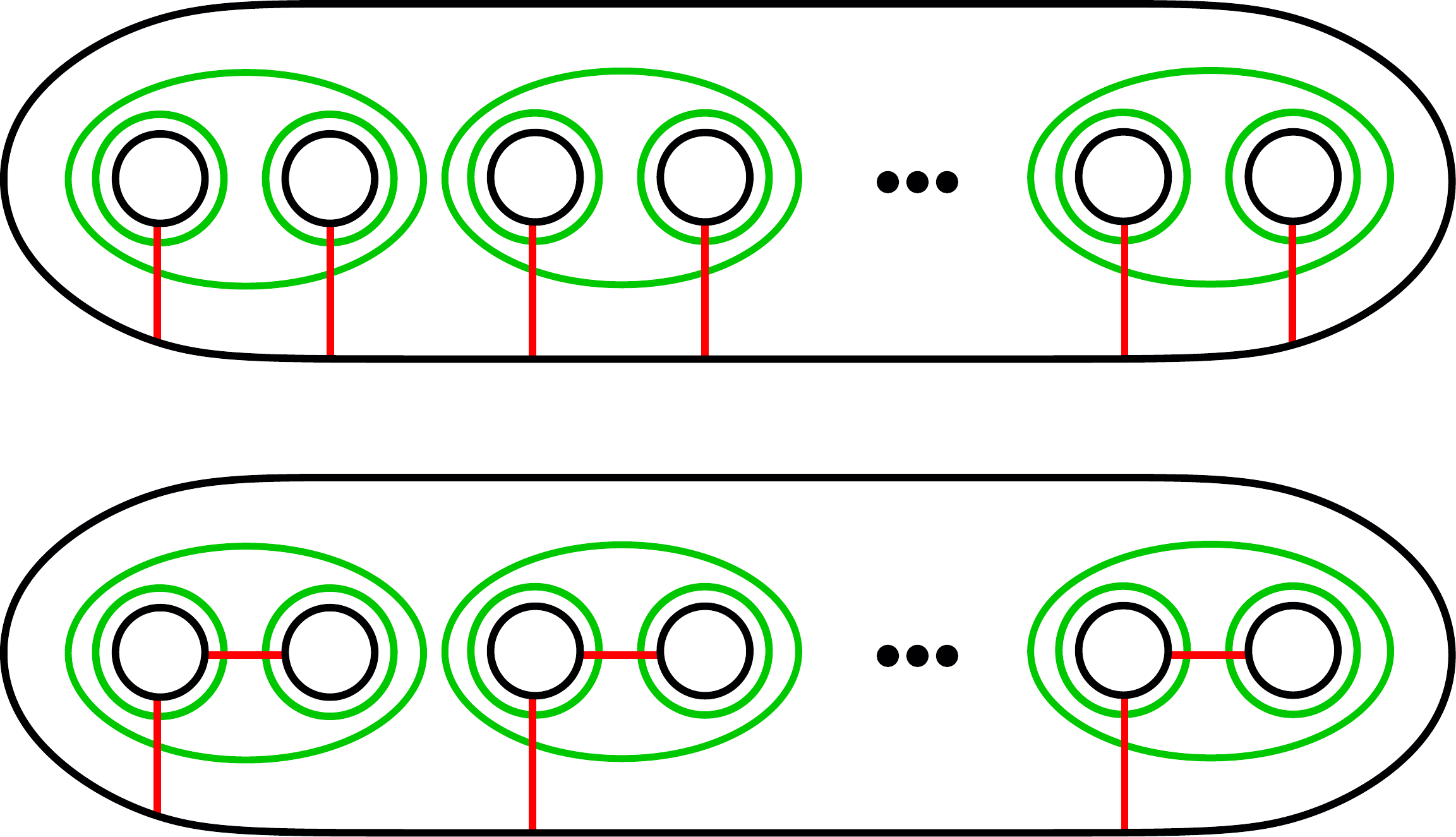}
		\caption{Fiber and vanishing cycles for the Lefschetz fibration of Example~\ref{ex:large1}, together with two arc systems, related through $n$ arc slides.}
		\label{fig:largeLF}
	\end{figure}
	
	Let $\Sigma_n=F_n\cup_{\partial F_n} \overline{F_n}$ and $d_n=\partial F_n$. Let $\mathfrak{M}_n$ be the multisection with divides compatible with $LF_n$.
	
	The Weinstein domain supported by $LF_n$ is a boundary sum of disk bundles over $S^2$ with Euler number $-3$. It has a Weinstein presentation  consisting of $n$ split Legendrian unknots which have each been positively stabilized once from the standard max $tb$ unknot as in Figure~\ref{fig:many_stabilized_unknot}.
\end{examples}

\begin{figure}
	\centering
	\includegraphics[scale=.2]{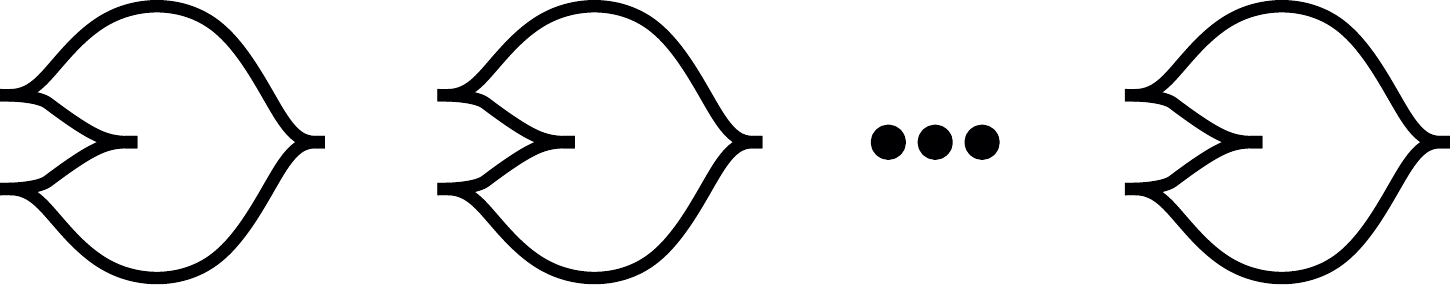}
	\caption{A Weinstein presentation for the Weinstein domain of Example~\ref{ex:large1}.}
	\label{fig:many_stabilized_unknot}
\end{figure}

\begin{proposition}\label{prop:arblargeL1}
	$\cL_{(\Sigma_n,d_n)}(\mathfrak{M}_n)=n$.
\end{proposition}

\begin{proof}
	We will show that any path in the contact cut graph $CC(\Sigma,d)$ with type $1$ moves realizing the vanishing cycles $\{a_1,b_1,c_1,\dots, a_n,b_n,c_n\}$, necessarily contains at least $n$ type $0$ moves. Note that since the vanishing cycles are disjoint, their order does not matter (Hurwitz moves do not change the curves).
	
	Each vertex in $CC(\Sigma,d)$ gives an arc system $\mathcal{A}^+$ for $F_n$ and another arc system $\mathcal{A}^-$ for $\overline{F_n}$. Each arc system consists of $2n$ arcs. Note that the endpoints of $\mathcal{A}^+$ are the same as those of $\mathcal{A}^-$. Because $F_n$ is planar, each arc in $\mathcal{A}^\pm$ has endpoints on two distinct boundary components, and each boundary component contains at least one endpoint of an arc.
	
	To perform a contact type $1$ move which corresponds to a Dehn twist about a curve $V\subset F_n$, the arc system $\mathcal{A}^+$ must be dual to $V$, meaning a single arc intersects $V$ in one point and the other arcs are disjoint. (Similarly for $\mathcal{A}^-$ if we perform the Dehn twist in $\overline{F_n}$.) Observe that if $V=a_i$, the arc systems $\mathcal{A}^\pm$ corresponding to the vertex before that type $1$ move must have exactly one arc with an endpoint on the $2i-1$ boundary component. Similarly if $V=b_i$, the preceding arc systems must have exactly one arc with an endpoint on the $2i$ boundary component. On the other hand, if $V=c_i$, the preceding arc systems must have exactly one arc which connects the $2i-1$ and $2i$ boundary components, and exactly one other arc which connects either the $2i-1$ or $2i$ boundary component to a different boundary component. (This follows from the facts above about arc systems for $F$ and the fact that $c_i$ separates the $2i-1$ and $2i$ boundary components from the others.)
	
	Observe that contact type $1$ moves do not change the endpoints of the arcs in the arc systems, so in particular they do not change which boundary components are connected by the arcs. Type $0$ moves can change this since they correspond to arc slides performed on both sides.

	Thus (regardless of the ordering of the vanishing cycles), there must be at least one type $0$ move in between the $a_i$ and $c_i$ type $1$ moves \emph{or} in between the $b_i$ and $c_i$ type $1$ moves. The type $0$ move must correspond to an arc slide exchanging an arc with one endpoint on boundary $2i$ and the other on $2i+1$ with an arc with exactly one of its endpoints on $2i$ or $2i+1$. Thus for each $i=1,\dots, n$, we must have at least one type $0$ move, so $\mathcal{L}_{(\Sigma_n,d_n)}(\mathfrak{M}_n)\geq n$.
	
	For the upper bound, we can find a path in $CC(\Sigma_n,d_n)$ starting with the arc system where there is a unique arc with an endpoint on boundary components $1,\dots, 2n$, and each of these arcs has their other end point on boundary component $0$ (see top of Figure~\ref{fig:largeLF}). We can then perform contact type $1$ moves which Dehn twist about $a_1,b_1,\dots, a_n,b_n$. Next we perform $n$ type $0$ moves, corresponding to the arc slides where the arc with endpoint on boundary component $2i-1$ slides over the arc with endpoint on boundary component $2i$. This provides an arc system (bottom of Figure~\ref{fig:largeLF}) where we can then perform contact type $1$ moves which Dehn twist about $c_1,\dots, c_n$.
\end{proof}

\begin{remark}
	The underlying Weinstein domains $(W_n,\omega_n,V_n)$ in these examples have non-vanishing $c_1(\omega)$ (this can be computed from the Weinstein Kirby diagram of figure~\ref{fig:many_stabilized_unknot}). Therefore by Corollary~\ref{cor:L0c1}, they cannot have Weinstein $\cL$-invariant equal to zero. However, to prove that $\cL(W_n,\omega_n,V_n)=n$ would require more sophisticated lower bounds. See questions~\ref{q:chern}, \ref{q:int}.
\end{remark}

\begin{examples} \label{ex:large2}
	Let $F_n$ be a planar surface with $n+1$ boundary components labeled $\{0,1,\dots, n\}$ for $n>1$. Let $V_0,V_1,\dots, V_n$ denote curves which are parallel to the boundary components. Let $LF_n$ denote the Lefschetz fibration with fiber $F_n$ and vanishing cycles $V_0,V_1,\dots, V_n$. The underlying Weinstein domain is a disk bundle over $S^2$ with Euler number $-(n+1)$. Let $\mathfrak{M}_n$ denote the multisection with divides compatible with $LF_n$, and $\Sigma_n=F_n\cup_{\partial F_n} \overline{F_n}$ and $d_n=\partial F_n$.
\end{examples}

\begin{figure}
	\centering
	\includegraphics[scale=.2]{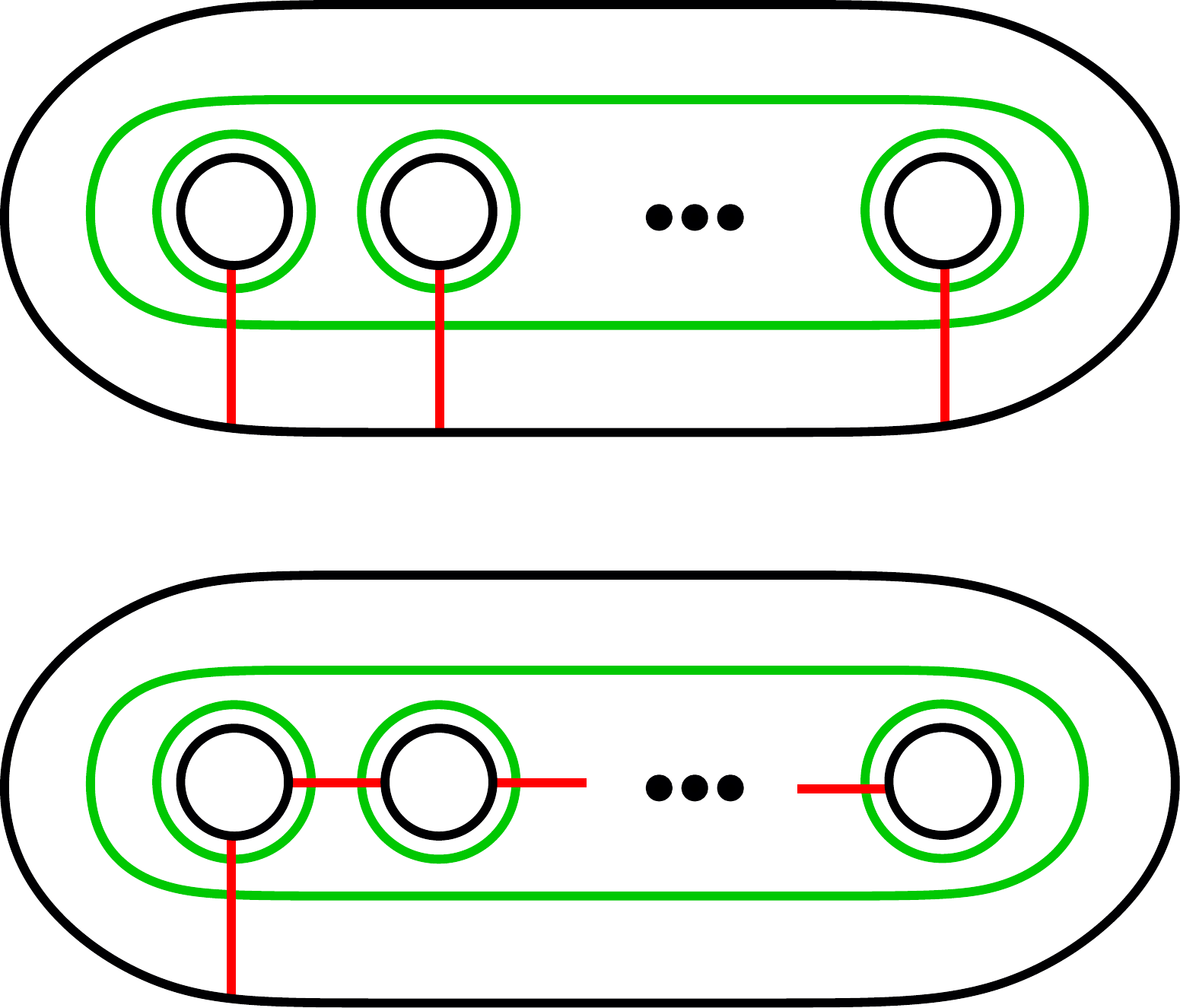}
	\caption{Fiber and vanishing cycles for the Lefschetz fibration of Example~\ref{ex:large2}, together with two arc systems, related through $n$ arc slides.}
	\label{fig:eg2}
\end{figure}

\begin{figure}
	\centering
	\includegraphics[scale=.2]{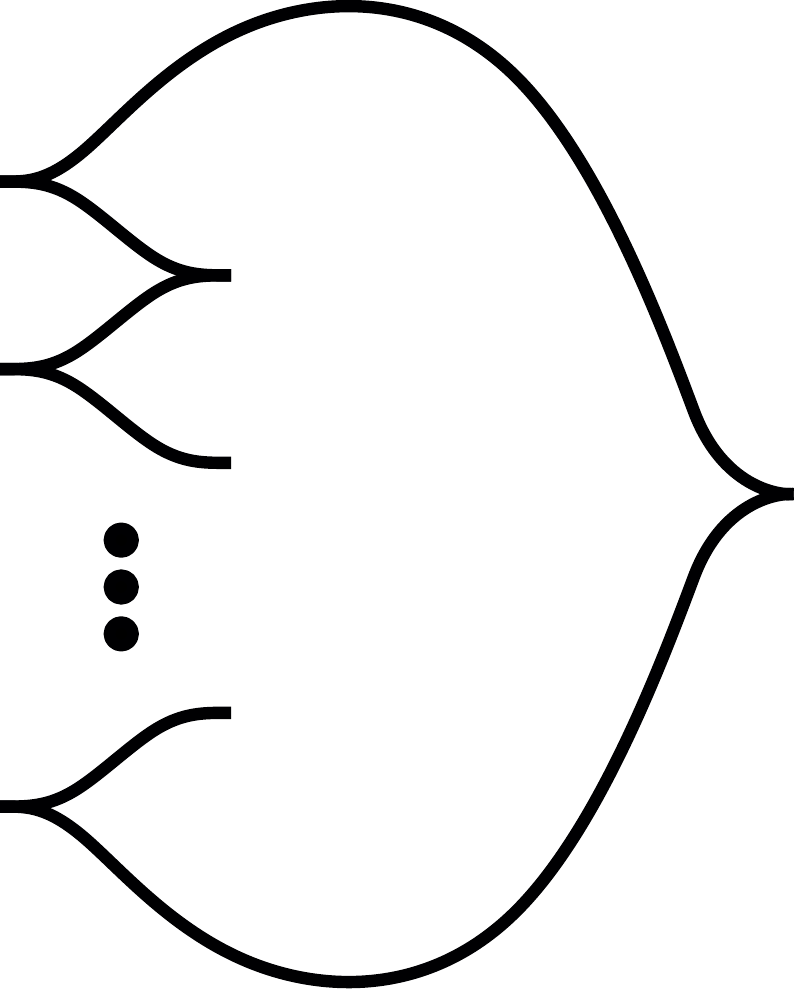}
	\caption{A Weinstein presentation for the Weinstein domain of Example~\ref{ex:large2}.}
	\label{fig:very_stabilized_unknot}
\end{figure}

\begin{proposition}\label{prop:arblargeL2}
	$\cL_{(\Sigma_n,d_n)}(\mathfrak{M}_n)=n-1$.
\end{proposition}

\begin{proof}
	We proceed similarly to the proof of Proposition~\ref{prop:arblargeL1}. In this case, before performing a contact type $1$ move corresponding to vanishing cycle $V_i$, the arc systems for $F_n$ and $\overline{F_n}$ must have a unique arc with an end point on the $i^{th}$ boundary component. Each arc system for $F_n$ has $n$ arcs, with a total of $2n$ endpoints. Each boundary component of $F_n$ contains at least one of these endpoints. This leaves $n-1$ ``extra'' endpoints to distribute among the $n+1$ boundary components. To perform the $V_i$ contact type $1$ move, we must first move any the ``extra'' endpoints of arcs off of the $i^{th}$ boundary component using type $0$ moves. Each type $0$ move, can move exactly one endpoint from one boundary component to another. Thus, to perform all $n+1$ type $1$ moves needed in $\mathfrak{M}_n$, we must perform at least $n-1$ type $0$ moves, since each ``extra'' endpoint must move at least once.
	
	To realize the path with exactly $n-1$ type $0$ moves, start with the double of an arc system where the $i^{th}$ arc connects boundary component $0$ to boundary component $i$ (top of Figure~\ref{fig:eg2}). Perform contact type $1$ moves on $V_1,\dots, V_n$. Then perform $n-1$ type $0$ moves corresponding to arc slides which lead to an arc system with arcs connecting boundary component $j-1$ with boundary component $j$ for $j=1,\dots, n$ (bottom of Figure~\ref{fig:eg2}). Finally, we can perform the contact type $1$ move for $V_0$.
\end{proof}

\begin{remark}
	In these examples, the Chern class is again non-vanishing when $n>1$ so $\cL\neq 0$. However, more work is needed to show that the underlying Weinstein domains have arbitrarily large Weinstein $\cL$-invariant. We leave this as an open conjecture.
\end{remark}

\begin{conjecture}
	The Weinstein domains in Example~\ref{ex:large1} and~\ref{ex:large2} have arbitrarily large Weinstein $\cL$-invariants, minimized over all possible supporting multisections with divides.
\end{conjecture}

We mention some quantities which appear to impact the Weinstein $\cL$ invariant, which could be explored to find lower bounds.

\begin{question}\label{q:chern}
	Can one prove a lower bound on $\cL(W,\omega,V)$ in terms of the first Chern class, for example $c_1(\omega)^2$?
\end{question}

\begin{question}\label{q:int}
	Can one prove a lower bound on $\cL(W,\omega, V)$ in terms of the intersection form of $W$?
\end{question}

Recall that Weinstein domains with $\cL=0$ have free fundamental group, and thus free $H_1$. For smooth closed 4-manifolds, Asano-Naoe-Ogawa proved a lower bound on the Kirby-Thompson invariant in terms of the order of $H_1$ when it is finite. A similar bound may hold in the Weinstein setting.

\begin{question}\label{q:H1}
	Can one formulate a lower bound on $\cL(W,\omega, V)$ in terms of the number of torsion elements in $H_1(W)$, or based on measurements of non-freeness of $\pi_1(W)$?
\end{question}

\bibliography{References}
\bibliographystyle{amsalpha}

\vspace{0.3in}

\end{document}